%% file: ms.tex
\documentclass[11pt,reqno]{amsart}
\usepackage[shortlabels]{enumitem}

\usepackage{tikz}

\usepackage[unicode]{hyperref}

\usepackage{latexsym, mathrsfs}
\usepackage{amsrefs}

\usepackage{algorithm}
\usepackage{algpseudocode}
\usepackage{listings}

\usepackage[breakable,skins]{tcolorbox}
\newtcolorbox{tbox}[1][]{%
    breakable,
    enhanced,
    colframe=black,
    coltitle=white,
    #1
}

\usepackage{amsmath,amsthm,amsfonts}
\usepackage{amssymb}
\usepackage{graphicx}
\usepackage{float}
\usepackage{caption}

\theoremstyle{plain}
\newtheorem{theorem}{Theorem}[section]
\newtheorem{lem}[theorem]{Lemma}

\newtheorem{cor}[theorem]{Corollary}

\newtheorem{computation}[theorem]{Computation}

\theoremstyle{definition}
\newtheorem{defn}[theorem]{Definition}
\newtheorem{eg}[theorem]{Example}

\newtheorem{rmk}[theorem]{Remark}

\newcommand{\bC}{\mathbb C}
\newcommand{\bZ}{\mathbb Z}
\newcommand{\bN}{\mathbb N}
\newcommand{\bQ}{\mathbb Q}
\newcommand{\bE}{\mathbb E}
\newcommand{\bF}{\mathbb F}
\newcommand{\bG}{\mathbb G}
\newcommand{\bP}{\mathbb P}

\newcommand{\LL}{\mathcal L}

\newcommand{\half}{\frac{1}{2}}

\newcommand{\Bl}{\text{Bl}}

\newcommand{\bEff}{\text{Eff}}
\newcommand{\bEffc}{\overline{\bEff}}

\newcommand{\Pic}{\text{Pic}}

\newcommand{\bFrob}{\text{Frob}}
\newcommand{\res}{\text{res}}

\newcommand{\gen}[1]{\langle {#1} \rangle}

\newcommand{\SL}{\text{SL}}

\begin{document}

\title{On The Geometry of Elliptic Pairs}
    \author[E.~Pratt]{Elizabeth Pratt}
    \email{epratt@berkeley.edu}
    \maketitle

    \begin{abstract}
    An \emph{elliptic pair} $(X, C)$ is a projective rational surface $X$ with log terminal singularities, and an irreducible curve $C$ contained in the smooth locus of $X$, with arithmetic genus $1$ and self-intersection $0$. They are a useful tool for determining whether the pseudo-effective cone of $X$ is polyhedral \cite{effcone}, and interesting algebraic and geometric objects in their own right. Especially of interest are toric elliptic pairs, where $X$ is the blow-up of a projective toric surface at the identity element of the torus. In this paper, we classify all toric elliptic pairs of Picard number two. Strikingly, it turns out that there are only three of these. Furthermore, we study a class of non-toric elliptic pairs coming from the blow-up of $\bP^2$ at nine points on a nodal cubic, in characteristic~$p$. This construction gives us examples of surfaces where the pseudo-effective cone is non-polyhedral for a set of primes $p$ of positive density, and, assuming the generalized Riemann hypothesis, polyhedral for a set of primes $p$ of positive density.
    \end{abstract}

	\section{Introduction} \label{intro}
	The effective cone of a projective variety $X$ and its closure, $\bEffc(X),$ are well-studied invariants of $X.$ In particularly nice cases $\bEffc(X)$ is polyhedral. This is true, for example, when $X$ is a projective toric variety. However, in general it is difficult to determine for an arbitrary projective surface $X$ whether $\bEffc(X)$ is polyhedral. Recent work by Castravet, Laface, Tevelev, and Ugaglia \cite{effcone} has shown that in the presence of a curve $C$ on $X$ satisfying certain properties, a polyhedrality criterion can be obtained in terms of the group structure of $\Pic (C)$. Using this criterion, they were able to prove that the Grothendieck-Knutsen moduli space $\overline{M_{0, n}}$ of stable rational curves has a non-polyhedral effective cone for $n \geq 10,$ by proving the corresponding statement for blow-ups of certain toric surfaces $\bP$. 
	
	More precisely, an \emph{elliptic pair} $(X, C)$ is a projective rational surface $X$ with log terminal singularities, and an irreducible curve $C$ contained in the smooth locus of $X$, such that the arithmetic genus of $C$ is one and $C^2 = 0$.  These elliptic pairs are not only useful for determining the polyhedrality of $\bEffc(X),$ but are also interesting geometric and arithmetic objects in their own right. In this paper we will construct elliptic pairs in two ways: first from blow-ups of toric surfaces of Picard number one, and then from blow-ups of $\bP^2$ at nine points on a nodal cubic.
	
	In Section \ref{triangles} we follow the procedure given in \cite{effcone} to construct \emph{toric elliptic pairs} $(X, C)$, which arise from lattice polygons $\Delta$ satisfying certain combinatorial properties. In this case $X$ will be the blow-up of $\bP(\Delta)$ at the identity element $e$ of the torus. We may denote $X$ instead by $X_\Delta$ if we wish to emphasize an underlying lattice polygon of $X$.
	
	In \cite{effcone} a number of interesting classification questions were raised. For instance, it is known that there are infinitely many pentagons which give rise to toric elliptic pairs with non-polyhedral effective cone. However, classification of quadrilaterals giving toric elliptic pairs is not known. Moreover, we do not have any examples of toric elliptic pairs $(X, C)$ where $X$ has Picard number three and $\bEffc(X)$ is non-polyhedral (see Remark 4.5 of \cite{effcone}).
	
	In this paper we completely classify all toric elliptic pairs coming from triangles, i.e.~where $X$ has Picard number two, the smallest possible. Strikingly, unlike the pentagon case, we prove that combinatorial restrictions allow only three such pairs.
	
	\begin{theorem} \label{trianglethm1}
	There are only three toric elliptic pairs $(X, C)$ where $X$ has Picard number two. They are given by lattice triangles $\Delta$ with vertices $\{(0, 0), (2,0), (5, 8)\}$ with $m = 4,$ $\{(0, 0),(5, 0), (12,20)\}$ with $m = 10,$ and $\{(0, 0),(5, 0), (18, 45)\}$ with $m = 15,$ where $m$ is the width of $\Delta.$
	\end{theorem}
	
	\begin{rmk}
	Since $X$ has Picard number two in all of these cases, the effective cone is two-dimensional and generated by $C$ and the exceptional divisor $E$ over the identity element of the torus. Thus in this case the polyhedrality question is trivial and we are really only interested in the problem of classifying these pairs.
	\end{rmk}
	
	In Section \ref{fibrations} we show for each lattice triangle $\Delta$ in Theorem \ref{trianglethm1} that $\Delta$ gives rise to an extremal elliptic fibration. We will compute the singular fibers and Kodaira type. Figure~\ref{fig:triangle1intro} shows the curves on the minimal resolution $\widetilde{X}$ of $X$ for the $m=4$ case. Here the underlying elliptic pair is $(X, C).$ The curve $C$, which is not pictured in Figure~\ref{fig:triangle1intro}, passes through each boundary divisor given by a side of $\Delta$ with multiplicity equal to the lattice length of that side.
	
	The curves $\tilde{g}$ and $\tilde{h}$ in Figure~\ref{fig:triangle1intro} are additional curves of self-intersection~$-1$ on $X,$ which are obtained by analyzing curves on $X$ as in will be described in Section \ref{fibrations}.
	
	\begin{figure}[ht]
        \centering
        \input{triangle1}
        \caption{Intersections of curves on $\widetilde{X_\Delta}$, where $\Delta = \{(0, 0), (2,0), (5, 8)\}$}
        \label{fig:triangle1intro}
    \end{figure}
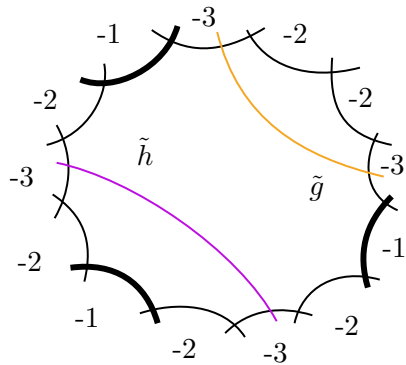
    
    By Castelnuovo's contraction criterion, we can then contract $\tilde{g}$ and $\tilde{h}$ to obtain the minimal model $Z$ of $\widetilde{X}.$ Figure~\ref{fig:singfibers} shows the singular fibers obtained for the $m=4$ case.
    
    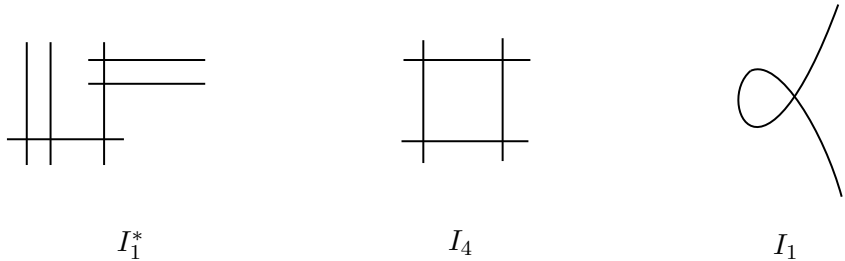
\begin{figure}[ht]
        \centering
        \input{singularfibers}
        \caption{The singular fibers of the minimal elliptic fibration $Z_\Delta,$ where $\Delta = \{(0, 0), (2,0), (5, 8)\}$}
        \label{fig:singfibers}
    \end{figure}
	
	Sections \ref{triangles} and \ref{fibrations} also answer a question posed by González-Anaya, González, and Karu in \cite{negcurves} and Kurano in \cite{kurano}: are there negative curves with positive genus on $\Bl_e \bP(\Delta)$ where $\Delta$ is a lattice triangle? A negative curve $C$ in $X$ is defined as a curve on $\Bl_e \bP(\Delta)$ of non-positive self-intersection. Our toric elliptic pairs $(\Bl_e\bP(\Delta)), C)$ provide an example of such curves, since by the definition of an elliptic pair we have $C^2 = 0$ and $p_a(C) = 1.$
	
	The paper \cite{effcone} also studies the distribution of primes $p$ such that the reduction of the toric elliptic pair $(X, C)$ modulo $p,$ where $C$ is an elliptic curve, has non-polyhedral effective cone. This is related to deep conjectures in arithmetic geometry of elliptic curves, including the Lang-Trotter conjecture \cite{langtrotter}. 
	
	In Section \ref{nodal}, we study an analogous question for a class of (non-toric) elliptic pairs, which come from blow-up of $\bP^2$ at nine points on a nodal cubic $C$. Instead of the Lang-Trotter conjecture, which concerns the arithmetic of elliptic curves, we will study connections to Artin's conjecture, which can be interpreted in terms of the arithmetic of nodal cubics. This construction gives us examples of surfaces where the pseudo-effective cone is non-polyhedral for a set of primes $p$ of positive density, and, assuming the generalized Riemann hypothesis, polyhedral for a set of primes $p$ of positive density.
    
	For our construction, we identify the smooth locus of the nodal cubic $C$ with $\Pic^0 (C),$ and $\Pic^0(C)$ with $\bG_m$ in such a way that $1 \in \bG_m$ is a flex point of $C.$ We say that $a, q \in \bQ$ are \emph{multiplicatively independent} in $\bQ^*$ if $a^xq^y =1$ implies that $x = y = 0.$ Let $z_1 = ... = z_7 = 1, z_8 = a, z_9 = qa^{-1}$ with $a, q \in \bQ$ multiplicatively independent. We construct $X$ to be the blow-up of $\bP^2$ at the nine points $z_1, ..., z_9$, which is an infinitely near blow up for $z_1, ..., z_7.$ That is, we consecutively blow up the point of intersection of the proper transform of $C$ with the exceptional divisor of the previous blow-up. Let $\bar{a}$ and $\bar{q}$ be the reductions of $a$ and $q$ modulo $p.$ 
	
	The surface $X$ defined this way and the proper transform of $C$ give an arithmetic elliptic pair $(\mathcal{C}, \mathcal{X})$ \cite{effcone}. That is, $(\mathcal{C}, \mathcal{X})$ are a pair of schemes which are flat over a nonempty open subset $\mathcal{U}$ of $\textup{Spec } \bZ,$ such that the the geometric fiber $(C_p, X_p)$ is an elliptic pair for every $p$ in $\mathcal{U}$. In this setting, we derive the following arithmetic condition for polyhedrality.
	
	\begin{theorem} \label{arithmeticintro}
	The pair $(\mathcal{C}, \mathcal{X})$ gives an arithmetic elliptic pair \cite{effcone} over a nonempty open subset $\mathcal{U}$ of $\textup{Spec } \bZ,$ such that the the geometric fiber $(C_p, X_p)$ is an elliptic pair for every $p$ in $\mathcal{U}$. Then $\overline{\textup{Eff}} (X_p)$  is polyhedral if and only if $\bar{a}^2 \in \gen{\bar{q}} \subset \bF_p^*.$
	\end{theorem}
	
	The question of when the condition in Theorem \ref{arithmeticintro} holds has been studied extensively in number theory. More precisely, consider the set of primes $p$ for which the condition is satisfied:
	\begin{equation*}
	    S(a, q) = \{ p \text{ prime: } \bar{a} \in \gen{\bar{q}} \subset \bF^*_p \}.
	\end{equation*}
	
	We show in Lemma \ref{nonpoly} that the set $S(a, q)^c$ contains a set of prime numbers of positive density. That is, there is a set of primes of positive density for which $\bEffc(X_p)$ is not polyhedral. 
	
	We also mention some known results about $S(a, q).$ First, by a theorem of P{\'o}lya \cite{polya}, $S(a, q)$ is infinite. Therefore, $\bEffc (X_p)$ is polyhedral for infinitely many primes. Furthermore, by Moree and Stevenhagen \cite{moree}, assuming the generalized Riemann hypothesis, $S(a, q)$ has positive density in the set of all prime numbers.
	
	Another useful result is Artin's conjecture on primitive roots, which states that if $q \in \bZ$ is neither a perfect square nor $-1,$ then $\bar{q} \mod p$ generates $\bF_p^*$ for a set of primes $p$ of positive density. If Artin's conjecture holds, we see that $S(a, q)$ will contain a set of positive density in the prime numbers.
	
	By work of Hooley \cite{hooley}, Artin's conjecture follows from the generalized Riemann hypothesis. 
	Thus by either Artin's conjecture or the work of Moree and Stevenhagen, if we accept the generalized Riemann hypothesis, there exists a set of primes of positive density for which $\bEffc(X_p)$ is polyhedral.
	
	Finally, a theorem of Heath-Brown tells us that for any three primes $(a, b, c),$ Artin's conjecture is true for at least one of $a, b,$ and $c$ \cite{heathbrown}. Combined with Lemma \ref{nonpoly}, this allows us to construct surfaces with a polyhedral effective cone for a set of primes of positive density, and a non-polyhedral effective cone for a set of primes of positive density. We construct such a surface concretely in Lemma \ref{threesurfaces}.
	
	\subsection{Acknowledgements} I am grateful to my senior thesis advisor, Jenia Tevelev, for his guidance and extensive feedback throughout this work. I would also like to thank Tom Weston for helpful discussions. This project has been partially supported by the NSF grant DMS-2101726 (PI Jenia Tevelev).
	
	\section{Toric Elliptic Pairs From Triangles} \label{triangles}
	We make preparations to define a toric elliptic pair, as in \cite{effcone}. First, given a lattice polygon $\Delta,$ we define a morphism
	\begin{align*}
	    g_{\Delta}:  \bG_2^m & \to \bP^{|\Delta \cap \bZ^2|-1}\\
	    x & \mapsto [x^iy^j: (i, j) \in \Delta \cap \bZ^2]
	\end{align*}
	We let $\bP(\Delta)$ be the closure of the image of $g_{\Delta},$ and $e = g_{\Delta}(1,1).$ We denote by $\LL_\Delta$ the linear system of hyperplane sections. Then any $f$ in $\LL_\Delta$ can be written as a Laurent polynomial with exponent vectors in $\Delta.$ That is, 
	\begin{equation*}
	    f = \sum_{(i, j) \in \Delta} a_{ij}x^iy^j \in k[x^{\pm 1}, y^{\pm 1}].
	\end{equation*}
	
	Given a positive integer $m$, we define $\LL_\Delta(m)$ to be the linear subspace of $\LL_\Delta$ consisting of the curves having multiplicity at least $m$ at $e$.
	Finally, we let $\text{Vol}(\Delta)$ be twice the Euclidean area of $\Delta,$ so that $\text{Vol}(\Delta)$ will be an integer.
	
	\begin{defn} \label{toricelliptic}
	Suppose there exist $\Delta$ and $m$ such that the following conditions hold:
	\begin{enumerate}[(i)]
	    \item $\textup{Vol}(\Delta) = m^2$;
	    \item $|\partial \Delta \cap \bZ^2| = m$;
	    \item There is an irreducible curve $\Gamma$ in $\LL_{\Delta}(m)$ such that 
	    \begin{enumerate}
	        \item $\Gamma$ has multiplicity $m$ at $e$
	        \item The Newton polygon of $\Gamma$  coincides with $\Delta.$
	    \end{enumerate}
	\end{enumerate}
	Let $C$ be the proper transform of $\Gamma$ under the blow-up. Then we call $(\textup{Bl}_e(\bP(\Delta), C)$ a \emph{toric elliptic pair}.
	\end{defn}
	
	\begin{rmk}
	Theorem 4.4 of \cite{effcone} proves that $(\Bl_e\bP(\Delta), C)$ is an elliptic pair. In particular, (i), (ii), and (iiia) tell us that $C$ has arithmetic genus 1, and (iiib) tells us that $\Gamma$ does not pass through the singularities of $\bP(\Delta).$ Thus $C$ does not pass through the singularities of $\Bl_e\bP(\Delta).$ Also, note that in the definition of toric elliptic pair given in \cite{effcone} the authors additionally require that $\dim \LL_{\Delta}(m) = 1.$ However, their proof of Theorem 4.4 does not use this fact, so we have dropped this assumption in Definition \ref{toricelliptic}.
	\end{rmk}

	 We call a lattice triangle \emph{primitive} if $\gcd(a, b, c) = 1.$ Now we can state our main theorem. In this section we will prove that there are at most three toric elliptic pairs of Picard number two, given by a primitive lattice triangle. In the next section we will verify that each of the lattice triangles in Theorem \ref{triangles} actually gives us a toric elliptic pair, and show that non-primitive triangles cannot give toric elliptic pairs. In the rest of this paper we will take ``triangle" to mean ``lattice triangle."
	
	As one may see from Definition \ref{toricelliptic}, not all lattice polygons give toric elliptic pairs. In Lemma \ref{widthlemma} we prove an additional arithmetic condition on the triangles $\Delta$ which do give toric elliptic pairs, in terms of the width of $\Delta$. Theorem \ref{trianglethm} shows that the arithmetic conditions from Lemma \ref{widthlemma} and (i) and (ii) of Definition \ref{toricelliptic} are only satisfied for three primitive lattice triangles, up to $\SL_2(\bZ)$ transformation.
	
	\begin{lem} \label{widthlemma}
	Let $\Delta$ be a lattice polygon, let $X = \textup{Bl}_e \bP (\Delta)$ and suppose $(X, C)$ is a toric elliptic pair. Let $m = | \partial \Delta \cap \bZ^2|.$ Then the width of $\Delta$ is $\geq m.$ 
	\end{lem}
	
	\begin{proof}
	Pick a primitive vector $\lambda$ in the lattice $\bZ^2$ of one-parameter subgroups. Let $D_{\lambda}$  be the prime divisor of $\bP (\Delta)$ given by the closure of the corresponding one-parameter subgroup. More explicitly, if $\lambda = (i, j)$ the corresponding one-parameter subgroup is $(t^i, t^j),$ which is given by the equation $x^j = y^i.$
	
	Let $H$ be a very ample divisor on $\bP(\Delta)$ that gives a linear system $\LL_\Delta.$ Let $w_\lambda$ be the width of $\Delta$ along $\lambda.$ Then $w_{\lambda} = H \cdot D_{\lambda}.$
	
	Now, consider the proper transform $\widetilde{D_{\lambda}}$ in the blow-up $\Bl_e \bP (\Delta).$ Since $D_{\lambda}$ goes through the identity with multiplicity one, we have $E \cdot \widetilde{D_{\lambda}} = 1.$  We also have $\pi^* H \sim C + mE.$
	
	By the projection formula, 
	$$w_{\lambda} = H \cdot D_{\lambda} = \pi^*H \cdot \widetilde{D_{\lambda}}.$$
	So we have
	\begin{align*}
	    w_{\lambda} & = (C + mE) \cdot \widetilde{D_{\lambda}}
	    = C \cdot \widetilde{D_{\lambda}} + m(E \cdot \widetilde{D_{\lambda}}) \\ 
	    & = C \cdot \widetilde{D_{\lambda}} + m.
	\end{align*}
	
	The Newton polygon of $C$ is $\Delta$ by (iii) of Definition \ref{toricelliptic}. However, the Newton polygon of $D_{\lambda}$ is a line segment. Thus they are not translates of one other. This tells us that the proper transforms of $C$ and $D_{\lambda}$ are two different irreducible curves on the blow-up. Thus $C \cdot \widetilde{D_{\lambda}} \geq 0,$ giving us that $w_{\lambda} \geq m.$
	\end{proof}
	
	We begin by choosing representatives for equivalence classes of polygons under transformations in $\text{SL}_2(\bZ).$ Every triangle is equivalent to a triangle with vertex set $\{(0,0), (a, 0), (b, c)\}$ such that $a > 0$ and $0 \leq b < c.$
	
	\begin{theorem} \label{trianglethm}
	Let $\Delta$ be a primitive lattice triangle with width $w.$ Let $m = | \delta \Delta \cap \bZ^2 |.$ Suppose
	\begin{enumerate}[(i)]
	    \item $w \geq m$;
	    \item $m^2$ is twice the Euclidean area of $\Delta.$
	\end{enumerate}
	
	Then $\Delta$ can be obtained via an $SL_2(\bZ)$ transformation from one of $\{(0, 0), (2,0), (5, 8)\}$ with $m = 4,$ $\{(0, 0),(5, 0), (12,20)\}$ with $m = 10,$ and $\{(0, 0),(5, 0), (18, 45)\}$ with $m = 15.$
	\end{theorem}
	
	\begin{proof}

	Note that by definition of the area of $\Delta$ we have $m^2 = ac.$  Also, the lattice length of the sides of $\Delta$ are $a, \gcd (b, c),$ and $\gcd (a - b, c).$ Thus $a + \gcd (b, c) + \gcd (a-b, c) = m.$ 
	
	By assumption, the width of $\Delta$ is at least $m$. Let $\lambda = (1, -1)$ and $\nu = (1, 0).$ Then we have $w_{\lambda} = a - (b - c) = a - b + c$ and $w_{\nu} = \max \{a, b\}.$ Then $a - (b - c) \geq m$ and $\max \{a, b\} \geq m.$ But we know $a < m,$ so $\max \{a, b\}  = b \geq m.$ Rearranging, we get that $m - a \leq c - b \leq c - m.$
	
	Let $b' = c - b.$ Observe that we can impose any ordering on the lattice lengths of the sides of $\Delta$. Then we have the following conditions, where \eqref{ordering} is the ordering of side lengths we impose:
	\begin{equation} \label{area}
	    m^2 = ac
	\end{equation}
	\begin{equation} \label{sidesum}
	    a + \gcd (b', c) + \gcd(b'+ a, c) = m 
	\end{equation}
	\begin{equation} \label{ordering}
	    a \geq \gcd (b', c) \geq \gcd(b' + a, c)
	\end{equation}
	\begin{equation} \label{widthconditions}
	    m - a \leq b' \leq c - m.
	\end{equation}
	In particular, by the ordering encoded by $\eqref{sidesum}$ and $\eqref{ordering}$, we have $\frac{a}{m} \geq \frac{1}{3}.$ We proceed by cases, considering different possible values of $\frac{a}{m}.$
	
	\vspace{0.3cm}
	\textbf{Case 1:} $\frac{a}{m} > \frac{2}{3}.$ 
	
	We will prove an upper bound for an expression involving $a, m,$ and $c.$ First, observe that $\frac{c - m}{m - a} = \frac{m}{a} < \frac{3}{2}.$ Thus by \eqref{widthconditions} we have 
	\begin{equation} \label{bconditions}
	    m-a \leq b' < \frac{3}{2}(m-a).
	\end{equation}
	
	Also, by \eqref{sidesum} and \eqref{ordering} we have $m-a > \gcd (b', c) \geq \frac{m-a}{2}.$ Since by \eqref{widthconditions} we have $b' \geq m-a,$ we cannot have $b' = \gcd (b', c).$ But by \ref{bconditions} we also have $b' < \frac{3}{2}(m-a) \leq 3\gcd (b', c).$ Thus $b' = 2 \gcd (b', c).$
	
	Let $d = \gcd (b', c).$ Then $b' = 2d$ and $c = kd$ for some odd $k$ in $\bZ.$ By \eqref{widthconditions} we have 
	\begin{align*}
	    \frac{d}{m-a} \leq \half \cdot \frac{c-m}{m-a} = \half \cdot \frac{m}{a} < \half \cdot \frac{3}{2} = \frac{3}{4}.
	\end{align*}
	
	Then by \eqref{sidesum}, $\gcd (b' + a, c) > \frac{1}{4} (m-a).$	On the other hand, $\gcd(a, d) = 1$ implies $ \gcd (b' + a, c) = \gcd (2d + a, k).$ Thus we have $k > \frac{m-a}{4}.$
	
	Also, from condition \eqref{sidesum} we have $d + \gcd (b' + a, c) = m-a.$ Combining this with \eqref{ordering} we get $d \geq \frac{m-a}{2}.$
	
	Thus we obtain 
	\begin{equation} \label{upperbdprep}
	    c = kd > \frac{(m-a)^2}{8} = \frac{1}{8} \cdot a  (c - 2m + a).
	\end{equation}
	
	We also have $ \frac{c}{a} = \left( \frac{m}{a} \right)^2 < \frac{9}{4}.$ Combining this with \eqref{upperbdprep} we see $\frac{1}{8} (c - 2m + a) < \frac{c}{a} < \frac{9}{4}.$ Clearing denominators we obtain our promised upper bound:
	\begin{equation} \label{upperbound}
	    c - 2m + a < 18.
	\end{equation}
	
	Now, consider the following parameterization of the equation $ac = m^2$: $a = es^2, c = et^2,$ and $m = ets,$ where $e, t, s \in \bZ$ and $\gcd(t, s) = 1$.
	Then our inequality \eqref{upperbound} becomes $(t-s)^2e < 18.$ This leaves finitely many pairs $(t-s, e)$ to consider. 
	
	Let $x = t-s.$ Then inequality \eqref{widthconditions} becomes 
	\begin{equation} \label{newwidthcond}
	    esx \leq 2d \leq e(s + x)x.
	\end{equation}
	
	Thus we can write $2d = esx + p$ where $0 \leq p \leq ex^2 < 18.$ From the equation $c = kd$ we get
	\begin{equation} \label{maineq}
	    2e(s+x)^2 = k(esx + p).
	\end{equation}
	
	Equation \eqref{maineq} gives us a family of diophantine equations in $s$ and $k$ with a finite set of parameters $(e, x, p).$ We will proceed by solving this family of diophantine equations abstractly in terms of $e, x,$ and $p,$ and using a computer to substitute particular values of $(e, x, p).$
	
	We rearrange equation \eqref{maineq} into quadratic form with respect to $s$, obtaining $2es^2 + ex(4-k)s + 2ex^2 - pk = 0.$
	Solving for $s$, we obtain
	\begin{align*}
	    s = \frac{1}{4e}(-ex(4-k) \pm \sqrt{D}), \qquad D:= e^2x^2(4-k)^2 - 8e(2ex^2 - pk).
	\end{align*}
	
	Since $-ex(4-k) \pm \sqrt{D}$ is in $4 \bZ,$ we see that $\sqrt{D}$ must be in $\bZ.$ Let $y = \sqrt{D}.$ Then we solve 
	\begin{align*}
	    y^2 & = D\\
	    & = e^2x^2(4-k)^2 - 8e(2ex^2 - pk)\\
	    & =(ex)^2k^2 - 8e(ex^2 - p)k.
	\end{align*}
	
	Completing the square with respect to $k$ gives us
	\begin{equation*}
	    \left( exk - \left( ex - \frac{p}{x} \right) \right)^2 = y^2 + 16 \left( ex - \frac{p}{x} \right)^2.
	\end{equation*}
	
	Multiplying by $x^2$ on each side to clear denominators gives us $(ex^2k - 4(ex^2 - p))^2 = y^2x^2 + 16(ex^2 - p)^2.$ We can then factor the difference of squares to obtain
	\begin{equation} \label{findky}
	    (ex^2k - 4(ex^2 - p) - xy)(ex^2k - (ex^2 - p) + xy) = 16(ex^2 - p)^2.
	\end{equation}
	
	Now we simply need to plug in values of $e, x,$ and $p$ such that $0 \leq p \leq ex^2 < 18$ and factor the right hand side of \eqref{findky}. We will obtain a finite list of options for $k$ and $y$ for each choice of $(e, x, p)$ by solving two linear equations in two variables. We can further subject $(k, y)$ to the conditions that $k$ is odd and that $\gcd(a, d) = 1.$ Each choice of $(k, y)$ will give us two options for $s$ (corresponding to $\pm \sqrt{D}$). Thus we can completely check all possible triangles.
	
	Finally, by condition \eqref{sidesum}, we have $m - a - d = \gcd(2d + a, c).$ Without this condition, we have five candidate triangles. With this condition added, we have none. Sage code which executes the algorithm just described is included in Section~\ref{code1}.

	\vspace{0.3cm}
	\textbf{Case 2:} $\frac{1}{2} < \frac{a}{m} \leq \frac{2}{3}.$
	
	Observe that $\frac{c - m}{m - a} = \frac{m}{a} < 2.$ Thus by condition \eqref{widthconditions} we have $m-a \leq b' < 2(m-a).$ Also, by conditions \eqref{sidesum} and \eqref{ordering} we have $m-a > \gcd (b', c) \geq \frac{m-a}{2}.$ Thus we have $b' = 2d$ or $b' = 3d.$
	
	We will use the bounds on $\frac{a}{m}$ to bound possible $k,$ where $k = \frac{c}{d}$ is as in case 1. 
	
	By conditions \eqref{sidesum} and \eqref{ordering} we obtain $\frac{m}{2} > d \geq \frac{m}{6}.$ Re-arranging and substituting $k = \frac{c}{d}$, we obtain 
	\begin{equation*}
	   6 \left( \frac{c}{m} \right) \geq k > 2 \left( \frac{c}{m} \right). 
	\end{equation*}
	
	Now recall that $m^2 = ac.$ Thus $\frac{a}{m} \cdot \frac{c}{m} = 1.$ But we also have $\frac{1}{2} < \frac{a}{m} \leq \frac{2}{3}.$ Thus $ 2 > \frac{c}{m} 
	\geq \frac{3}{2}.$ Substituting these inequalities into our bounds on $k,$ we obtain
	\begin{equation} \label{kbounds}
	   12 > k > 3. 
	\end{equation}
	
	Suppose $b' = 2d.$ Since $(b', c) = d,$ we know $k$ must be odd. Thus we obtain a finite list of options: $k = 5, 7, 9, 11.$ 
	
	Next, suppose $b' = 3d.$ Then $k$ cannot be divisible by 3. Thus we obtain a finite list of options: $k = 4, 5, 7, 8, 10, 11.$ 
	
	Furthermore, by condition \eqref{sidesum} and $\gcd(a, d) = 1,$ we have that $d = m - a - \gcd(2d + a, k).$ Let $x = \gcd(2d + a, k).$ Then $x$ must divide $k.$ Thus we obtain a finite family of diophantine equations in $(a, m)$ with parameters $(k, x)$:
	\begin{equation} \label{areatri2}
	    m^2 = ac = kda = ka(m - a - x).
	\end{equation}
	
	Putting equation \eqref{areatri2} into quadratic form with respect to $\frac{m}{a}$ and solving gives us
	\begin{equation} \label{case2quadratic}
	    2 \left(\frac{m}{a}\right) = k \pm \sqrt{k^2 - 4k  - 4 \frac{kx}{a}}.
	\end{equation}
	
	But recall we have $\frac{1}{2} < \frac{a}{m} \leq \frac{2}{3},$ so 
	\begin{equation} \label{rootbounds}
	    3 \leq 2 \left( \frac{m}{a} \right) < 4.
	\end{equation}
	
	Since $k \geq 4$ by equation \eqref{kbounds} we have $k + \sqrt{k^2 - 4k  - 4 \frac{kx}{a}} \geq 4.$ Thus the larger root in \eqref{case2quadratic} will always fall outside the bound given in \eqref{rootbounds}, and we can restrict our attention to the smaller root.
	
	Suppose that $k = 4.$ Then equation \eqref{kbounds} has no real solutions. Thus we must have $k > 4.$
	
	Since $k>4$ we have that $2k - 9 > 0.$ Now, suppose that $a > \frac{4kx}{2k - 9}.$ Then $2k - 4 \frac{kx}{a} >9.$ Adding $k^2$ to both sides, we get $k^2 - 6k + 9 < k^2 - 4k - 4 \frac{kx}{a}.$ Thus $k-3 < \sqrt{k^2 - 4k  - 4 \frac{kx}{a}}.$ Finally, we see $3 > k - \sqrt{k^2 - 4k  - 4 \frac{kx}{a}},$ which falls outside of our bound in \eqref{rootbounds}. Thus if \eqref{rootbounds} if satisfied, we must have.
	\begin{equation} \label{case2condition}
	    a \leq \frac{4kx}{2k - 9}.
	\end{equation}
	
	We thus have a finite family of triples $(k, a, x),$ which we can use to solve the equation $m^2 = ka(m-a-x)$ for possible $m.$ We can then check for each $m$ obtained this way that $m$ is in an integer, and that $\frac{a}{m} > \frac{1}{2}.$ The code for this procedure is included in Section~\ref{code2}. 
	
	Running the code results in the candidates:
	\begin{align*} 
	    & k = 5,  \ a = 9, \ x = 1, \ m = 15\\
	    & k = 5, \ a = 45, \ x = 5, \ m = 75.
	\end{align*}
	
	Finally, we determine whether candidate triangles $(m, a, c)$ listed above are primitive lattice triangles satisfying the conditions \eqref{area} through \eqref{widthconditions}.
	
	\textbf{Candidate 1:} We calculate $d = m - a - \gcd(2d + a, k) = 5.$ Then $c = kd = 5 \cdot 5 = 25.$ Next, suppose $b' = 2d.$ Then $b = c - 2d = 25 - 2 \cdot 5 = 15.$ Thus we get the triangle $\{(0,0), (9, 0), (15, 25)\}.$ This triangle is $\text{SL}_2(\bZ)$ equivalent to $\{(0, 0),(5, 0), (18, 45)\},$ and the remaining conditions can be easily checked. 
	
	If we instead suppose that $b' = 3d,$ we get $b = 10.$ Thus we get the triangle $\{(0, 0), (9, 0), (10, 25)\}.$ However, this does not satisfy $\eqref{widthconditions},$ as $10 = b \ngeq m = 15.$
	
	\textbf{Candidate 2:} We calculate $d = m - a - \gcd(2d + a, k) = 25.$ Thus $c = kd = 125.$ Then $b' = 2d$ or $b' = 3d.$ Thus $b = 50$ or $75.$ In either case, the lattice triangle we obtain is not primitive, i.e. $\gcd (a, b, c) > 1.$ 
	
	
	\vspace{0.3cm}
	\textbf{Case 3:} $\frac{1}{3} \leq \frac{a}{m} \leq \frac{1}{2}.$ 
	
	Observe that $\gcd (b, c) = \gcd(b', c) = d.$ Let $b = b_0d$ and $c = c_0d.$ Recall that by condition \eqref{ordering}, we have $d \geq \gcd (b'+a, c).$ So by condition \eqref{sidesum} we get that $d \geq \frac{m-a}{2}.$ But since $a \leq \frac{m}{2}$ we get $d \geq \frac{m}{4}.$ Also, using $a \geq \frac{m}{3}$ and condition \eqref{area}, we get $c = \left( \frac{m}{a} \right)m \leq 3m.$ Thus we have 
	\begin{align*}
	    c_0d \leq 3m = 12 \left( \frac{m}{4} \right) \leq 12d.
	\end{align*}
	
	We also have that $b \geq m > 2d.$ Thus $b_0 \geq 3.$ To summarize, we've obtained the finite list of possibilities
	\begin{align*}
	     & 3 \leq b_0 < c_0 \leq 12.
	\end{align*}
	
	From condition \eqref{widthconditions} we also have $c_0 - b_0 \geq \frac{m}{d}.$ From condition \eqref{ordering} we have $m \geq 2a \geq 2d,$ so $\frac{m}{d} \leq 2.$ So $c_0 - b_0 \geq 2.$
	
	Let $x = \gcd (b' + a, c)\gcd (b - a, c).$ Since $1 = \gcd (a, b, c) = (a, d),$ we get
	\begin{align*}
	    x = \gcd (b_0d - a, c_0d) = \gcd(b_0d - a, c_0) | c_0.
	\end{align*}
	
	Thus we have a finite list of possible $x.$ By condition \eqref{sidesum}, we have $d = m - a - x.$ From this and condition \eqref{ordering} we get that $m \geq 2d = 2(m - a - x).$ Simplifying, we obtain
	\begin{equation*}
	    \frac{m}{2} \geq a \geq \frac{m-2x}{2}.
	\end{equation*}
	
	Thus we can consider the finite list $a = \frac{m}{2}, ..., \frac{m - 2x}{2}.$ Suppose $a = \frac{m-y}{2}$ where $0 \leq y \leq 2x.$ We also have $c = c_0d = c_0 (m - a - x).$ Substituting for $a$ and $c$ into the equation $ac = m^2,$ we obtain
	\begin{equation*}
	    \left( \frac{m - y}{2} \right) c_0 \left( m -  \frac{m - y}{2} - x \right) = m^2.
	\end{equation*}
	
	Simplifying, we obtain the following quadratic equation in $m$:
	
	\begin{equation} \label{case3quadratic}
	    (c_0 - 4) m^2 - 2c_0xm - c_0y(y - 2x) = 0.
	\end{equation}
	
	We can solve equation \eqref{case3quadratic} for $m$ for each of a finite family of parameters $(b_0, c_0, x, y)$ subject to the conditions $0 \leq y \leq 2x$, $x$ divides $c_0,$ $3 \leq b_0 \leq c_0 \leq 12,$ and $c_0 - b_0 \geq 2.$  Then we check whether the $m$ obtained is an integer and whether the lattice triangle is primitive. The code for the procedure described in this paragraph is given in Section \ref{code3}. Finally, we check that our candidate triples $(a, b, c)$ satisfy $\eqref{area}$ through \eqref{widthconditions}.
	
	
	By running the code in Section \ref{code3}, we obtain the triples 
	\begin{align*}
	    & m = 3: (1, 5, 9)\\
	    & m = 4: (2, 3, 8), (2, 5, 8) \\
	    & m = 6: (3, 8, 12) \\
	    & m = 10: (5, 12, 20).
	\end{align*}
	We eliminate $(1, 5, 9),$ since the width of the corresponding triangle is $2.$ We eliminate $(2, 3, 8),$ since the width of the corresponding triangle is $3.$ We eliminate $(3, 8, 12)$ because the number of lattice points on the boundary is $8.$ The other two triples satisfy all of conditions \eqref{area} through \eqref{widthconditions}.
	
	\begin{eg}
	We include an example of how the code in Section \ref{code3} will proceed for a particular choice of $(b_0, c_0, x, y)$.Suppose $b_0 = 5$ and $c_0 = 8.$ Then $x = 1, 2, 4, 8$ are possibilities. Equation \eqref{case3quadratic} gives us
	\begin{align*}
	    & 4m^2 - 16xm + 8(2xy-y^2) = 0.
	\end{align*}
	Suppose $x = 1$ and $y = 0.$ Then we obtain $m^2 - 4m = 0.$ Thus $m = 4.$ Then $a = \frac{m}{2} = 2.$ To calculate $d,$ we solve $4^2 = m^2 = c_0da = 8\cdot d \cdot 2.$ Thus $d = 1.$ 
	
	The code produces the candidate triple $(2, 5, 8).$ We manually check the lattice perimeter: $a + d + x = 2 + 1 + 1 = 4.$ Indeed, this triangle is one of the three triangles in the statement of Theorem \ref{trianglethm}.
	\end{eg}
	\end{proof}
	
	\section{Elliptic Fibrations} \label{fibrations}
	\begin{defn}
	An elliptic fibration is a morphism from an irreducible projective surface $S$ to a smooth curve, e.g. $\bP^1$, such that a general fiber is an elliptic curve. An elliptic fibration is called \emph{extremal} if it has a section and the Mordell-Weil group of sections is finite.
	\end{defn}
	
	\begin{theorem}
	Each of the three lattice triangles with vertex sets $\{(0, 0), (2,0), (5, 8)\},$ $\{(0, 0), (5, 0), (12,20)\},$  and $\{(0,0), (5, 0), (18, 45)\}$ of Theorem \ref{trianglethm} gives us a rational extremal elliptic fibration. The corresponding minimal rational elliptic fibrations are surfaces of types $X_{141}, X_{211},$ and $X_{211}$ respectively \cite{artebani}. That is, they have singular fibers of types $I_1^*I_4I_1, II^*I_1I_1,$ and $II^*I_1I_1$ in Kodaira's classification.
	\end{theorem}
	
	\begin{proof}
    We show that if $\Delta$ is one of the three triangles above, then $\LL_{\Delta}(m)$ is an elliptic fibration, i.e. a pencil which contains an elliptic curve. We then compute the minimal resolutions of each $\Bl_e \bP(\Delta)$ and find the minimal model of the corresponding smooth rational elliptic fibration, and compute its Kodaira type. Since $\Delta$ has a side of lattice length one, the toric boundary divisor corresponding to this side intersects the elliptic curve in exactly one point. Thus we have at least one section of the fibration. Table 1 in \cite{artebani} then shows that the elliptic fibration is extremal. 
    
    Suppose that $\Delta$ is one of our three triangles. Then $m$ is equal to the width of $\Delta.$ Let $(v, -u)$ be a vector which achieves $m.$ Then the curve $(x^uy^v-1)^m$ given by $m$ copies of the one parameter subgroup has multiplicity $m$ at $e.$ Now, let $s= b-mu$ and consider the curve $g(x, y) = x^s(x^uy^v-1)^m,$ which also has multiplicity $m$ at $e.$
    
    We claim that the exponents of monomials in $\text{Supp } g$ are in $\Delta.$ To see this, write $$g(x, y) = \sum_{i=1}^{m} (-1)^ix^s(x^uy^v)^i.$$
    So all of the exponents of monomials in $g$ lie along the line from $(s, 0)$ to $m \cdot (u, v) + (s, 0) = (b, mv).$ In each of our examples, $(a, 0) \cdot (v, -u) = m,$ so $c = \frac{m^2}{a} = \frac{amv}{a} = mv.$
    Thus the endpoints of the line segment are in the convex polygon $\Delta,$ so the entire line segment lies in $\Delta.$ We conclude that $g$ is in $\LL_{\Delta} (m).$ We will use the existence of $g$ to find curves on $\Bl_e \bP(\Delta)$ which we can contract to compute its minimal resolution.
    
    \begin{eg} \label{geg}
    Consider the triangle $\Delta = \{(0, 0), (2,0), (5, 8)\}$ with $m = 4$. By Computation~\ref{genus1} below, the width of $\Delta$ is achieved in the direction $(2, -1).$ We expand $g(x, y) = x(xy^2-1)^4 = x((xy^2)^4 - (xy^2)^3 + (xy^2)^2 - xy^2 + 1).$ One can verify visually that each of these monomials have exponents in $\Delta$, as in Figure~\ref{triangle1width}. 
    
    \begin{figure}[ht] 
    \centering
	\begin{tikzpicture}
    \foreach \x in {0,1,2,3,4,5, 6, 7, 8}
       \foreach \y in {0,1,2,3,4, 5, 6, 7, 8} 
          \draw[fill] (2/4*\x,2/4*\y) circle (0.5pt) coordinate (m-\x-\y);
    
    \draw (m-0-0) -- (m-2-0) -- (m-5-8) -- cycle;
    \draw[red] (m-1-0)--(m-5-8);
    \end{tikzpicture}
    \caption{The triangle $\Delta$ for $m=4.$ All monomials in $g$ are lattice points along the median.}
    \label{triangle1width}
    \end{figure}
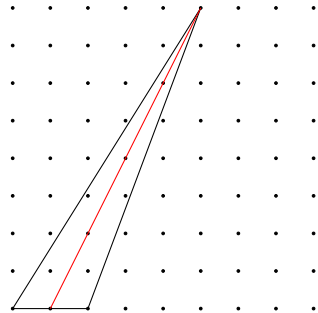
    \end{eg}
    
    Next, we compute using the Magma package ``non-polyhedral" which is available on Github \cite{lafacepkg} that for each of our surfaces $\bP(\Delta),$ there is an curve $\Gamma$ in $\LL_{\Delta}(m)$  whose Newton polytope coincides with $\Delta.$ We include the $m = 4$ case as an example.
    
    \begin{computation} \label{genus1}
    Let $\Delta = \{(0, 0), (2,0), (5, 8)\}.$ We find a curve $\Gamma$ on $X_\Delta$ which goes through the identity element of the torus with multiplicity $4$. Our $\Gamma$ is denoted $f$ in the computation below. Here the command ``Genus" gives the geometric genus of $f.$
    \begin{tbox}
    {\footnotesize
    \begin{verbatim}
    > pol:=Polytope([[0,0],[2,0],[5,8]]);
    > Width(pol);
    4 {
        (2, -1),
        (-2, 1)
    }
    > f:=FindCurve(pol, 4, Rationals());
    > f;
    Curve over Rational Field defined by
    9/4*x[1]^5*x[2]^8 - 8*x[1]^4*x[2]^6 + 15/2*x[1]^3*x[2]^4 + 2*x[1]^3*x[2]^3 +
        2*x[1]^2*x[2]^3 - 6*x[1]^2*x[2] + x[1]^2 - 6*x[1]*x[2] + 17/4*x[1] + 1
    > Genus(f);
    1 \end{verbatim}
    }
    \end{tbox}
    \end{computation}
    We know that $p_a(C) = 1$ by Proposition 4.2 of \cite{effcone}. Thus we have that $C$ is smooth, irreducible, and genus one, hence an elliptic curve.
    
    We also need to check that $\Gamma$ does not pass through the singularities of $\bP(\Delta).$ As observed in Section \ref{triangles}, this is equivalent to the Newton polytope of $\Gamma$ being equal to $\Delta.$ Again, we include the $m = 4$ case as an example.
    
    \begin{computation} \label{npolytope1}
    We compute the Newton polytope of $\Gamma.$
    \begin{tbox}
    {\footnotesize
    \begin{verbatim}
    > pol:=Polytope([[0,0],[2,0],[5,8]]);
    > f:=FindCurves(pol, 4, Rationals())[1];
    > Transpose(Matrix(Vertices(NPolytope(f))));
    [5 2 0]
    [8 0 0] \end{verbatim}
    }
    \end{tbox}
    \end{computation}
    
    Let $X = \Bl_e\bP(\Delta)),$ and $C$ be the proper transform of $\Gamma$. Then both $C$ and the proper transform $\tilde{g}$ of $g$ are in $H^0(X, C)$ by Claim \ref{gclaim}, and they are not the same curve, since one is an elliptic curve and the other is a rational curve. Thus for each of our surfaces, $h^0(X, C) \geq 2.$ By Lemma 3.2 in \cite{effcone}, $h^0(X, C) = 2.$ Thus $(X, C)$ gives an elliptic fibration. 
     
    To show that $(X, C)$ is extremal, we compute the minimal resolution $\widetilde{X}$ of $X,$ and the minimal elliptic fibration $Z$ of $\widetilde{X}.$ We recover the Kodaira type of $Z,$ showing that the fibration is extremal.
     
    \vspace{0.3cm}
    \textbf{Triangle 1:} $\Delta = \{(0, 0), (2,0), (5, 8)\}$ with $m = 4$.
    
    We compute the minimal resolution of $\widetilde{X}$ of $X,$ which is given by Figure~\ref{fig:triangle1}. The number on each curve in Figure~\ref{fig:triangle1} indicates its self-intersection. The lines in bold are the curves of self intersection $-1$ given by the sides of $\Delta.$ 
    
    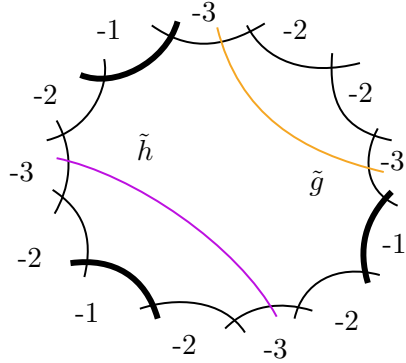
\begin{figure}[ht]
        \centering
        \input{triangle1}
        \caption{Intersections of curves on $\widetilde{X_\Delta}$, where $\Delta = \{(0, 0), (2,0), (5, 8)\}$}
        \label{fig:triangle1}
    \end{figure}
    
    The curve $\tilde{g}$ in Figure~\ref{fig:triangle1} is the proper transform of $g(x, y) = x(xy^2-1)^4.$ The curve $\tilde{h}$ is the proper transform of the curve $h(x, y) = x^2y^3 - 3xy + x + 1.$  We obtain $h$ by factoring the adjoint linear system $|K + C|$ into prime components as in Computation~\ref{factorkc}, where $K$ is the canonical divisor on $X.$
    
    \pagebreak
    \begin{computation} \label{factorkc}
    Finding the factorization of $|K + C|$ for $\Delta = \{(0, 0), (2,0), (5, 8)\}.$
    \begin{tbox}
    {\footnotesize
    \begin{verbatim}
    > pol := Polytope([[0,0],[2,0],[5,8]]);
    > PolsAdjSys(pol);
    [x[1]*x[2]^2 - 1, x[1]^2*x[2]^3 - 3*x[1]*x[2] + x[1] + 1]\end{verbatim}
    }
    \end{tbox}
    \end{computation}
    
    By adjunction, $(K+C) \cdot C = 0,$ so $\tilde{h}$ and $C$ are disjoint. In Computation~\ref{selfint} we check using Magma that $\tilde{g}$ and $\tilde{h}$ are curves of self-intersection $-1$ on $\widetilde{X}.$ This will later allow us to contract $\tilde{g}$ and $\tilde{h}$ using Castelnuovo's contraction criterion.
    
    \begin{computation} \label{selfint}
    Computing the self-intersection of $C,$ $\tilde{g},$ and $\tilde{h}$ on $\widetilde{X_\Delta}.$ The rows in the matrix return by AdjSys represent the classes of $C,$ $\tilde{g},$ and $\tilde{h},$ respectively, in the Picard group of $\widetilde{X_\Delta}$. The last number in each row gives the multiplicity of the curve at $e.$
    \begin{tbox}
    {\footnotesize
    \begin{verbatim}
    > adjsysmatrix := Matrix(Rationals(), #AdjSys(pol), #AdjSys(pol)[1], AdjSys(pol));
    > adjsysmatrix;
    [ 8  5  2  1  0  0  0  0  0  0  1  2  5 -4]
    [ 2  1  0  0  0  0  0  0  0  0  0  0  1 -1]
    [ 3  2  1  1  1  0  0  0  0  1  1  1  2 -2]
    > matrix := imatS(pol);
    > adjsysmatrix*matrix*Transpose(adjsysmatrix);
    [ 0  0  0]
    [ 0 -1  0]
    [ 0  0 -1]\end{verbatim}
    }
    \end{tbox}
    \end{computation}
    
    We also use Magma to calculate the intersections between $\tilde{g}, \tilde{h},$ and the proper transforms of the toric boundary divisors of $X_\Delta,$ which are depicted in Figure~\ref{fig:triangle1}.
    
    \begin{computation} \label{imats}
    The intersection matrix of the minimal resolution $\widetilde{X_\Delta}$ of $X_\Delta.$
    \begin{tbox}
    {\footnotesize
    \begin{verbatim}
    > Transpose(Matrix(Reorder(Rays(Resolution(NormalFan(pol))))));
    [ 0 -1 -2 -5 -8 -3 -1  1  3  8  5  2  1]
    [ 1  1  1  2  3  1  0 -1 -2 -5 -3 -1  0]
    > imatS(pol);
    [-1  1  0  0  0  0  0  0  0  0  0  0  1  0]
    [ 1 -2  1  0  0  0  0  0  0  0  0  0  0  0]
    [ 0  1 -3  1  0  0  0  0  0  0  0  0  0  0]
    [ 0  0  1 -2  1  0  0  0  0  0  0  0  0  0]
    [ 0  0  0  1 -1  1  0  0  0  0  0  0  0  0]
    [ 0  0  0  0  1 -3  1  0  0  0  0  0  0  0]
    [ 0  0  0  0  0  1 -2  1  0  0  0  0  0  0]
    [ 0  0  0  0  0  0  1 -2  1  0  0  0  0  0]
    [ 0  0  0  0  0  0  0  1 -3  1  0  0  0  0]
    [ 0  0  0  0  0  0  0  0  1 -1  1  0  0  0]
    [ 0  0  0  0  0  0  0  0  0  1 -2  1  0  0]
    [ 0  0  0  0  0  0  0  0  0  0  1 -3  1  0]
    [ 1  0  0  0  0  0  0  0  0  0  0  1 -2  0]
    [ 0  0  0  0  0  0  0  0  0  0  0  0  0 -1]\end{verbatim}
    }
    \end{tbox}
    \end{computation}
    
    Each ray of the normal fan corresponds to a divisor in $\widetilde{X_\Delta}$. The toric boundary divisors corresponding to the sides of $\Delta$ are given by $(0, 1), (8, -3),$ and $(8, -5)$ in the normal fan, which correspond to the curves of self-intersection $-1$ in the $14 \times 14$ intersection matrix given by Computation \ref{imats}. We see from the exponents of $g$ that $\tilde{g}$ intersects the curves given by $\pm (2, -1)$ in the normal fan, which correspond to rows $6$ and $9$ in the intersection matrix, thus allowing us to place $\tilde{g}$ in Figure~\ref{fig:triangle1}.
    
    We can also see from the exponents of $h$ that $\tilde{h}$ intersects the curves in the minimal resolution which given by $(2, 3)$ and $(1, 1)$ in the normal fan, thus allowing us to place it in Figure~\ref{fig:triangle1}.
    
    The elliptic curve $C$, which is not pictured in Figure~\ref{fig:triangle1}, passes through each boundary divisor given by a side of $\Delta$ with multiplicity equal to the lattice length of that side.
    
    Since $\tilde{g}$ and $\tilde{h}$ each have self-intersection~$-1$ on $X_{\Delta}$,  we can contract them using Castelnuovo's contraction criterion. We obtain a configuration of $-2$ curves with Dynkin diagram $D_5$ from contracting $\tilde{g}$, and $A_3$ by contracting $\tilde{h}$, which correspond to singular fibers of Kodaira type $I_1^*$ and $I_4.$ These fibers are pictured in Figure~\ref{fig:singfibers}. Since we have an elliptic fibration, the sum of the Euler characteristics of the singular fibers must be 12. So we must also have a singular fiber of type $I_1.$ Thus the Kodaira type is $X_{141},$ and consequently the elliptic fibration is extremal \cite{artebani}.

     
     \vspace{0.3cm}
     \textbf{Triangle 2:} $\Delta = \{(0, 0), (5,0), (12, 20)\}$ with $m = 10$.
     
     As in Computation \ref{imats}, we obtain Figure~\ref{fig:triangle2} by computing the normal fan and intersection matrix of of $\Delta.$ By a Magma calculation, the width of $\Delta$ is achieved in the direction $(2, -1).$  The curve $\tilde{g}$ in Figure~\ref{fig:triangle2} is the proper transform of $g(x, y) = x^2(xy^2-1)^{10}.$
     
    \begin{figure}[ht]
        \centering
        \input{triangle2}
        \caption{Intersections of curves on $\widetilde{X_\Delta}$, with $\Delta = \{(0, 0), (5,0), (12, 20)\}$}
        \label{fig:triangle2}
    \end{figure}
    
    The curve $\tilde{h}$ in Figure~\ref{fig:triangle2} is the proper transform of $h(x, y) = x^2y^3 - 3xy + x + 1.$ As in the Triangle 1 case, $h(x, y)$ is a curve on $\bP(\Delta)$ which is obtained by factoring the adjoint linear system $|K + C|$ (see Computation \ref{factorkc2}). 
    
    \begin{computation} \label{factorkc2}
    Finding the factorization of $|K + C|$ for $\Delta = \{(0, 0), (5,0), (12, 20)\}.$
    \begin{tbox}
    {\footnotesize
    \begin{verbatim}
    > pol := Polytope([[0,0],[5,0],[12,20]]);
    > PolsAdjSys(pol);
    [x[1]*x[2]^2 - 1, x[1]^2*x[2]^3 - 3*x[1]*x[2] + x[1] + 1]\end{verbatim}
    }
    \end{tbox}
    \end{computation}
    The computation that $\tilde{g}$ and $\tilde{h}$ are disjoint and each have self-intersection $-1$ is identical to that of Computation \ref{selfint}, with the vertices of ``pol" replaced by the vertices of Triangle 2.
    
    Again, we can contract $\tilde{g}$ and $\tilde{h}$ using Castelnuovo's contraction criterion. We obtain a configuration of $-2$ curves with Dynkin diagram $E_8$ by contracting $\tilde{g}$. Similarly, we obtain a configuration of curves with Dynkin diagram $A_1$ by contracting $\tilde{h}$. These correspond to singular fibers of Kodaira type $II^*$ and $I_1,$ respectively. Thus the Kodaira type is $X_{211},$ and consequently the elliptic fibration is extremal \cite{artebani}.
     
    \vspace{0.3cm}
    \textbf{Triangle 3:} $\Delta = \{(0, 0), (5,0), (18, 45)\}$ with $m = 15$.
     
    By a Magma calculation, the width of $\Delta$ is achieved in the direction $(3, -1).$ The curve $\tilde{g}$ in Figure~\ref{fig:triangle3} is the proper transform of $g(x, y)  = x^3(xy^5-1)^{15}.$ As in Computation \ref{imats}, we obtain Figure~\ref{fig:triangle2} by computing the normal fan and intersection matrix of of $\Delta.$
    
     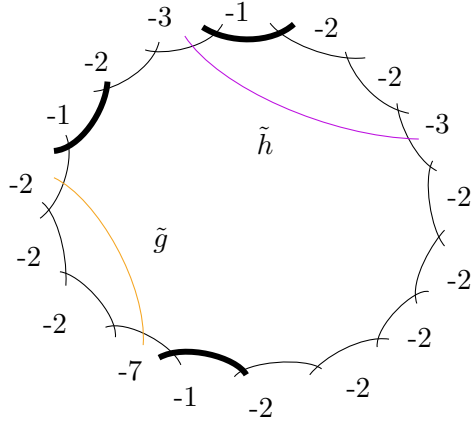
\begin{figure}[ht]
        \centering
        \input{triangle3}
        \caption{Intersections of curves on $\widetilde{X_\Delta}$, with $\Delta = \{(0, 0), (5,0), (18, 45)\}$}
        \label{fig:triangle3}
    \end{figure}
    The curve $\tilde{h}$ in Figure~\ref{fig:triangle3} is the proper transform of $h(x, y) = x^3y^7 - 2x^2y^5 - x^2y^4 + 5xy^2 - 3xy + x - 1,$ which is obtained by factoring the adjoint linear system $|K + C|$ (see Computation \ref{factorkc3}). The computation that $\tilde{g}$ and $\tilde{h}$ are disjoint and each self-intersection $-1$ is identical to that of Computation \ref{selfint}, with the vertices of ``pol" replaced by the vertices of Triangle 3.
    
    \begin{computation} \label{factorkc3}
    Finding the factorization of $K + C$ for $\Delta = \{(0, 0), (5,0), (18, 45)\}.$
    \begin{tbox}
    {\footnotesize
    \begin{verbatim}
    > pol := Polytope([[0,0],[5,0],[18,45]]);
    > PolsAdjSys(pol);
    [x[1]*x[2]^3 - 1, x[1]^3*x[2]^7 - 2*x[1]^2*x[2]^5 - x[1]^2*x[2]^4 + 5*x[1]*x[2]^2 
        - 3*x[1]*x[2] + x[1] - 1]\end{verbatim}
    }
    \end{tbox}
    \end{computation}
    
    Again, we can contract $\tilde{g}$ and $\tilde{h}$ using Castelnuovo's contraction criterion. We obtain $E_8$ by contracting $\tilde{g}$, and $A_1$ by contracting $\tilde{h}$, which correspond to singular fibers of Kodaira type $II^*$ and $I_1,$ respectively. Thus the Kodaira type is $X_{211},$ as in Triangle 2, and consequently the elliptic fibration is extremal \cite{artebani}.
	\end{proof}
	
	Now that we know that $(X_\Delta, C)$ are elliptic pairs for each of our triangles $\Delta$,  we can combine this result with Theorem \ref{trianglethm} to completely classify all toric elliptic pairs of Picard number two. 
	
	\begin{theorem} \label{classificationthm}
	Suppose that $(X, C)$ is a toric elliptic pair, where $X = \textup{Bl}_e(\bP(\Delta)$ for some lattice triangle $\Delta.$ Then $\Delta$ is one of $\{(0, 0), (2,0), (5, 8)\}$ with $m = 4,$ $\{(0, 0),(5, 0), (12,20)\}$ with $m = 10,$ and $\{(0, 0),(5, 0), (18, 45)\}$ with $m = 15.$
	\end{theorem}
	
	\begin{proof}
	First, recall that $\SL_2(\bZ)$-equivalent polygons give rise to isomorphic toric surfaces. Thus Theorem \ref{trianglethm} tells us that if $\Delta$ is primitive, then $\Delta$ can be transformed via an $\SL_2(\bZ)$ transformation to one of the triangles listed in the theorem. 
	
	Write $\Delta = \{(0, 0), (a, 0), (b, c)\}$ where $(\gcd(a, b, c) = k$ and $m = |\delta \Delta \cap \bZ^2|.$ Then $m = kn$ for some $n \in \bN.$ 
	
	Let $\Delta/k$ be $\{(0, 0), (a/k, 0), (b/k, c/k)\}.$ Then $\Delta / k$ is primitive and satisfies the hypothesis of Theorem \ref{trianglethm}. So $\Delta/k$ is one of the three triangles listed in the theorem.
	
	Since each of our triangles $\Delta/k$ gives an elliptic fibration, we know that $\dim \LL_{\Delta/k}(n) =~ 2.$ Let $\{g, h\}$ be a basis of $\LL_{\Delta/k}(n).$ Then $g^k$ and $h^k$ are two linearly independent curves in $\LL_\Delta(m),$ so $\dim \LL_\Delta(m) \geq 2.$ Thus by Lemma-Definition 3.2 of \cite{effcone}, we have that $\dim \LL_\Delta(m)=~2.$ Therefore, every member of $\LL_\Delta(m)$ is a linear combination of $g^k$ and $h^k$, and thus factors nontrivally over $\bC.$ But there must be an irreducible curve in $\LL_\Delta(m)$ by definition of a toric elliptic pair. Thus we have $k = 1$ and $\Delta / k = \Delta.$
	\end{proof}
	
	\pagebreak
	\section{Elliptic Pairs From a Nodal Cubic in $\bP^2$} \label{nodal}
	
	In this section, we will consider non-toric elliptic pairs coming from blow-ups of $\bP^2$ at nine points on the nodal cubic over an algebraically closed field of prime characteristic~$p$. 
	
	Let $C$ be the nodal cubic $y^2z = x^2(x+z)$ in $\bP^2.$ We will identify the smooth locus of $C$ with $\Pic^0 (C),$ and $\Pic^0(C)$ with $\bG_m.$ We choose $1 \in \bG_m$ to be the flex point $[0:1:0]$ of $C.$ We say that $a, q \in \bQ$ are \emph{multiplicatively independent} in $\bQ^*$ if $a^xq^y =1$ implies that $x = y = 0.$ 
	
	Let $z_1 = ... = z_7 = 1, z_8 = a, z_9 = qa^{-1}$ with $a, q \in \bQ$ multiplicatively independent. Let $X$ be the blow-up of $\bP^2$ at the nine points $z_1, ..., z_9$, which is an infinitely near blow up for $z_1, ..., z_7.$ That is, we consecutively blow up the point of intersection of the proper transform of $C$ with the exceptional divisor of the previous blow-up. Let $\bar{a}$ and $\bar{q}$ be the reductions of $a$ and $q$ modulo $p.$ Then $(C, X)$ is an elliptic pair defined over $\bQ.$ Thus by Definition 3.19 of \cite{effcone}, $(C, X)$ gives rise to an \emph{arithmetic elliptic pair} $(\mathcal{C}, \mathcal{X})$ over a nonempty open subset $\mathcal{U}$ of $\textup{Spec } \bZ$, where the geometric fiber $(C_p, X_p)$ is an elliptic pair for every $p$ in $\mathcal{U}$.
	
	\begin{theorem} \label{nodalthm}
	Let $(\mathcal{C}, \mathcal{X})$ be the arithmetic elliptic pair \cite{effcone} given by the procedure above, and defined over a nonempty open subset $\mathcal{U}$ of $\textup{Spec } \bZ,$ such that the the geometric fiber $(C_p, X_p)$ is an elliptic pair for every $p$ in $\mathcal{U}$. Then $\overline{\textup{Eff}} (X_p)$  is polyhedral if and only if $\bar{a}^2 \in \gen{\bar{q}} \subset \bF_p^*.$
	\end{theorem}
	
	\begin{proof}
	There is a restriction map
	\begin{equation*}
	    \res: \Pic(X) \to \Pic (C)
	\end{equation*}
	which sends $C^{\perp}$ into $\Pic^0(C).$ Also, observe that $C^{\perp} / \gen{C}$ with an intersection pairing gives rise to the root lattice $\bE_8.$ In particular, the curves of self-intersection $-2$ on $X$ are effective roots in $\bE_8$. We will show that a certain subset of these effective roots generates a root sublattice of rank 7.
	
	Let $E_i$ be the exceptional divisor of the $i$th blow-up. For ease of notation, we will allow ourselves to denote the proper transform of $\widetilde{E_i}$ after the $(i+2)$th blow-up as $\widetilde{E_i}$, because $\widetilde{E_i}$ is disjoint from $E_{i+2}$ and is thus unchanged by the $(i+2)$th blow-up. Thus $E_i = \widetilde{E_i} + .. . + \widetilde{E_{6}} + E_{7}$ for $i \leq 7.$ Then $E_i^2 = -1$ and $E_i \cdot E_j = 0$ if $|j-i| \geq 1.$ 
	
	Let $h$ be the class of a line in $\Pic^0(C)$. We can label the Dynkin diagram of $\bE_8$ with roots as in Figure~\ref{dynkin}, all of which are effective except $E_7 - E_8$.
	
    \tikzset{every picture/.style={line width=0.75pt}}
    
    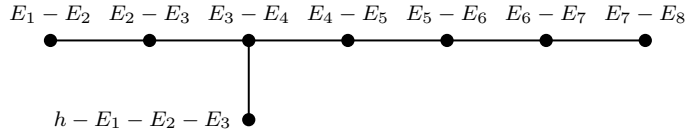
\begin{figure}[ht]
        \centering
     \begin{tikzpicture}[x=0.75pt,y=0.75pt,yscale=-1,xscale=1, dot/.style = {circle, fill = black, minimum size=#1, inner sep=0pt, outer sep=0pt}, dot/.default = 5pt, every label/.append style={ font=\scriptsize}]
     
    \draw    (100, 0) -- (150, 0) ;
    \node[dot, label={$E_1-E_2$}, sloped, shift={(100,0)}] at (0, 0) {};
    \node[dot, label={$E_2-E_3$}, shift={(150,0)}] at (0, 0) {};
    \draw    (150, 0) -- (200, 0) ;
    \node[dot, label={$E_3-E_4$}, shift={(200,0)}] at (0, 0) {};
    \draw    (200, 0) -- (250, 0) ;
    \node[dot, label={$E_4-E_5$}, shift={(250,0)}] at (0, 0) {};
    \draw    (250, 0) -- (300, 0) ;
    \node[dot, label={$E_5-E_6$}, shift={(300,0)}] at (0, 0) {};
    \draw    (300, 0) -- (350, 0) ;
    \node[dot, label={$E_6-E_7$}, shift={(350,0)}] at (0, 0) {};
    \draw    (350, 0) -- (400, 0) ;
    \node[dot, label={$E_7 - E_8$}, shift={(400,0)}] at (0, 0){};
    \draw    (200, 0) -- (200, 40) ;
    \node[dot, label={$h - E_1-E_2-E_3$}, shift={(200,-40)}, rotate = 90] at (0, 0) {};
    
    \end{tikzpicture}
    \caption{Dynkin diagram of $\bE_8$ labeled with simple roots}
    \label{dynkin}
    \end{figure}
    
    Since $1$ is a flex point of $C$, the root $(h - E_1 - E_2 - E_3)$ is realized by the class of a proper transform of a tangent line at $1$, so it is indeed an effective root. Also, $E_i - E_{i+1} = \widetilde{E_i}$ are effective roots for all $1 \leq i \leq 7.$ Computing the self-intersections, we see that $(E_i - E_{i+1})^2 = -2.$ Also, since $h$ is disjoint from $E_1, E_2,$ and $E_3,$ we have $(h - E_1 - E_2 - E_3)^2 = h^2 - 2h(E_1 + E_2 + E_3) + (E_1 + E_2 + E_3)^2 = 1 + 0 - 3 = -2.$ Finally, curves that correspond to adjacent nodes have intersection 1: $(E_i - E_{i+1})(E_{i+1} - E_{i+2}) = - E_{i+1}^2 = 1.$

	Because $\res (C^{\perp})$ lies in $\Pic^0(C),$ we obtain an induced map 
	
	\begin{equation*}
	    \overline{\res}: C^{\perp}/\gen{C} \to \Pic^0(C) / \gen{\res(C)}.
	\end{equation*}
	
	Now, consider the root sublattice $\bE_7$ of $\bE_8$ generated by $E_1 - E_2, ..., E_6- E_7$ and $h - E_1 - E_2 - E_3$. The roots in $\bE_7$ are mapped to $0$ by $\overline{\res},$ because the corresponding curves are disjoint from $C.$ 
	
	Furthermore, we can contract the effective roots in $\bE_7$ to get a surface $Y$ with a single Du Val singularity, as shown in Figure~\ref{fig:contraction}.
	
	 \begin{figure}[ht]
        \centering
        \scalebox{.8}{\input{surfacexy}}
        \caption{Contracting the $\bE_7$ sublattice in $X$ to obtain $Y$}
        \label{fig:contraction}
    \end{figure}
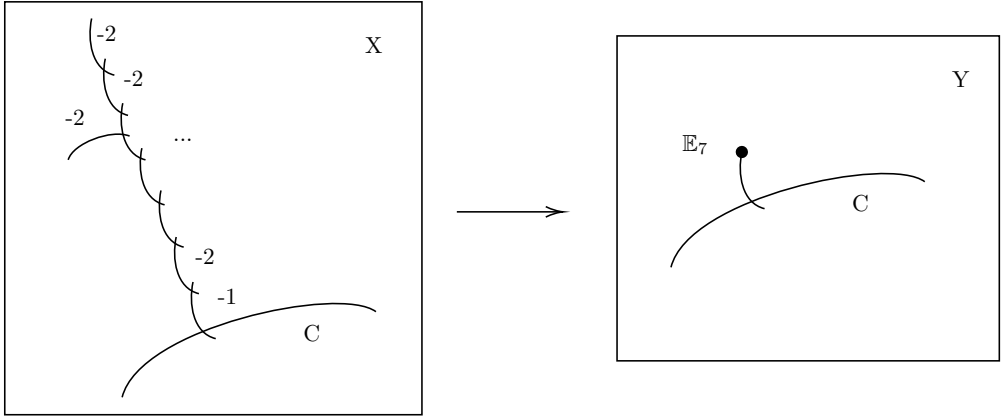
    From Corollary 3.14 in \cite{effcone}, $\bEffc(X)$ is polyhedral if and only if $\bEffc(Y)$ is polyhedral. Since $p(X) = 10,$ we see $p(Y) = 3.$  This allows us to use Corollary 3.18 in \cite{effcone} to obtain a criterion for polyhedrality in our case: $\bEffc(Y)$ is polyhedral if and only if $\overline{\res}(\beta) = 0$ for some root $\beta \in \bE_8 \setminus \bE_7.$ 
	
	We will interpret the condition $\overline{\res}(\beta) = 0$ in terms of $a$ and $q.$ First, we compute $\res(C)$ as follows:
	\begin{align*}
	    & \res(C) = \res(3h - \sum_{i = 1}^9 E_i) = 3\cdot \res(h - E_1 - E_2 - E_3) + 2(E_1 + E_2 + E_3) - \sum_{i = 4}^9 E_i.
	\end{align*}
	
	Since $h - E_1 - E_2 - E_3$ is represented by a curve of self-intersection $-2$ disjoint from $C,$ we have $\res(h - E_1 - E_2 - E_3) = 0.$ Restricting each $E_i$ to $C$ gives $E_i \cap C = [z_i],$ i.e.~the class of the point $z_i$ in $\Pic^0(C).$ Thus 
	\begin{equation*}
	    \res(C) = 2([z_1] + [z_2] + [z_3]) - \sum_{i = 4}^9 [z_i].
	\end{equation*}
	
	From our identification of $\Pic^0(C)$ with $\bG_m,$ we have
	\begin{equation*}
	    \res(C) = (z_1z_2z_3)^2(z_4...z_9)^{-1} = (aqa^{-1})^{-1} = q^{-1}.
	\end{equation*}
	Thus $\overline{\res}(\beta) = 0$ if and only if $\res(\beta) \in \gen{q}$ in $\bF_p^*.$ In Table \ref{roottable}, we analyze the image of the remaining roots of $\bE_8 = C^{\perp} / \gen{C}$ which are not in $\bE_7,$ again using the identification with $\bG_m.$ Note that since $1$ is a flex point, we can identify $h$ with the tangent line to $C$ at $1 = z_0$. Then $\overline{\res}(h) = 3[z_0].$
	
	\begin{table}[h!]
    \centering 
    \begin{tabular}{c c } 
     Roots &  Image in $\Pic^0(C) / \gen{\res(C)}$ \\[0.5ex] 
    \hline 
    $E_i - E_8, \ i < 8$ & $a^{- 1}$\\
    $E_i - E_9, \ i<9$ & $a$\\
    $E_8 - E_9$ & $a^2$\\ 
    $h - E_i - E_j - E_8, \ i < j < 8$ & $a^{-1}$ \\
    $h - E_i - E_j - E_9, \ i < j < 8$ & $a$ \\
    \hline
    \end{tabular}
    \caption{Images of roots under the map $\overline{\res}$}
    \label{roottable} 
    \end{table}
    
	Thus $\bEffc (X)$ is polyhedral if and only if at least one of $a$ and $a^2$ is in $\gen{q}.$ But if $a$ is in $\gen{q},$ then so is $a^2.$
	\end{proof}
	
	Next, we prove a lemma which tells us that the polyhedrality condition of Theorem \ref{nodalthm} fails for a set of primes of positive density. This lemma is known to experts, but we could not find a reference.
	
	\begin{lem} \label{nonpoly}
	Let $a, q \in \bQ$ be multiplicatively independent in $\bQ^*$. Consider the set 
	\begin{equation*} \label{primecondition}
	    S(a, q)^c = \{ p \text{ prime: } \bar{a} \notin \gen{\bar{q}} \text{ in } \bF_p^* \}.
	\end{equation*}
	Then $S(a, q)^c$ contains a set of prime numbers of positive density.
	\end{lem}
	
	\begin{proof}
	For ease of notation, we switch to writing $a, q$ for the reductions of $a, q$ modulo $p.$ Observe that $a \in \gen{q}$ if and only if $\gen{a} \in \gen{q}$ in $\bF_p^*.$ Let $o(x)$ denote the order of $x$ in $\bF_p^*.$ Since $\bF_p^*$ is cyclic, $\gen{a}$ is in $\gen{q}$ if and only if $o(a)$ divides $o(q),$ if and only if $[\bF_p^* : \gen{q}] \mid [\bF_p^* : \gen{a}].$
    
    Now, suppose $l \in \bZ$ is a prime satisfying the conditions (i) $l  \mid  [\bF_p^* : \gen{q}]$ and (ii) $l \nmid [\bF_p^{*}  : \gen{a}].$ Then $[\bF_p^* : \gen{q}] \nmid [\bF_p^* : \gen{a}].$ Thus (i) and (ii) are sufficient conditions for a prime $p$ to lie in $S(a, q)^c.$The conditions (i) and (ii) can be rewritten by multiplying both sides by $o(q)/l$, $o(a)/l$ respectively to get (i) $q^\frac{p-1}{l} = 1$ in $\bF_p^*$ and (ii) $a^\frac{(p-1)}{l} \neq 1$ in $\bF_p^*.$
    
    Let $K = \bQ(\zeta_l, a^{\frac{1}{l}}, q^{\frac{1}{l}}).$ We will proceed by encoding (i) and (ii) as conditions on the class of the Frobenius element of $p$ in the Galois group of $K$ over $\bQ$, and invoking Chebotarev's density theorem. 
    
    By the assumption that $a$ and $q$ are multiplicatively independent in $\bQ^*,$ we have $(a, q) = 1$ and that $a, q$ are not $l$th powers. Then the Galois group $G$ of $K / \bQ$ is isomorphic to $\bZ_{l}^{\times} \ltimes (\bZ_{l} \times \bZ_{l}).$ Let $\zeta$ denote the $l$th root of unity. Then $G$ is generated by elements of the form 
    \begin{align*}
        & \sigma_s: \zeta \mapsto \zeta^s\\
        & \tau_t: q^{\frac{1}{l}} \mapsto \zeta^t q^{\frac{1}{l}}\\
        & \rho_r: a^{\frac{1}{l}} \mapsto \zeta^r a^{\frac{1}{l}}.
    \end{align*}
    where if the action of $\sigma \in G$ on an element $\alpha \in K$ is not written, then $\sigma$ acts as the identity on $\alpha.$ We can thus identify $G$ with $\gen{\sigma_1} \ltimes (\gen{\tau_1} \times \gen{\rho_1}).$ 
    
    We proceed by writing $\bFrob_p = (\sigma_s, \tau_t, \rho_r)$ and determining what $s, t,$ and $r$ must be. By definition of the Frobenius element, we have $\bFrob_p(\zeta) = \zeta^p.$ But by condition (i) we have $p = 0 \mod l.$ So if $\bFrob_p(\zeta) = \zeta^s,$ then $s = 0.$ Similarly, we have $\bFrob_p(q^\frac{1}{l}) = q^\frac{p}{l}.$ Suppose that condition (i) holds, i.e.~$q^\frac{p-1}{l} = 1$ in $\bF_p^*.$ Then $\bFrob_p(q^\frac{1}{l}) = q^\frac{1}{l}.$ So $t = 0.$ By condition (ii), $\bFrob_p(a^\frac{1}{l}) \neq a^\frac{1}{l}.$ So $r \neq 0.$ Thus $p$ satisfies (i) and (ii) if and only if in the Galois group $G,$ $\bFrob_p$ is of the form $(\sigma_1, \tau_0, \rho_{\neq 0}).$ 
    
    We check that these elements form a conjugacy class. Any element $(\sigma_{s_0}, \tau_{t_0}, \rho_{r_0})$ in $G$ is fixed by by elements of the form $\tau_t, \rho_r.$ Conjugating by an element of the form $\sigma_s$ gives $(\sigma_{s_0}, \tau_{st_0}, \rho_{sr_0}).$ But since $1 \leq s \leq p-1,$ we remain in the conjugacy class.
    
    Applying the Chebotarev density theorem, we see that there are asymptotically $(l-1) / ((l-1) \cdot l \cdot l)$ such primes.
	\end{proof} 
	
	In general, it is difficult to ``patch" together different values of $l$ to obtain an asymptotic density for $S(a, q)^c,$ since the probabilities given by the Chebatorev density theorem are not statistically independent \cite{moree}. However,  Moree and Stevenhagen showed in Theorem 2 of \cite{moree} that assuming the generalized Riemann hypothesis, the density of $S(a, q)$ can be precisely calculated to be
	\[c_{a, q} \cdot \prod_{p \text{ prime}} \left( 1 - \frac{p}{p^3 -1}\right)\]
	where $c_{a,q}$ is a constant depending on $a$ and $q.$ The product $\prod_{p \text{ prime}} \left( 1 - \frac{p}{p^3 -1}\right)$ converges to the Stephens constant $S$, whose approximation to $50$ decimal places is
	\[S \approx 0.57595 99688 92945 43964 31633 75492 49669 25065 13967 17649.\]
	In the case that $\bQ^*/\gen{-1, a, b}$ is torsion-free, then $c_{a, b}$ is between $0.981$ and $1.024.$ Thus in this case $(X_p, C_p)$ has polyhedral effective cone for roughly $57\%$ of primes, and non-polyhedral effective cone for roughly $43\%$ of primes.
	
	Furthermore, a theorem of Heath-Brown tells us that for any three primes $(a, b, c),$ Artin's conjecture is true for at least one of $a, b,$ and $c$ \cite{heathbrown}. Combined with Lemma \ref{nonpoly}, this allows us to explicitly construct surfaces with a polyhedral effective cone for a set of primes of positive density, and a non-polyhedral effective cone for a set of primes of positive density.
	
	\begin{cor} \label{threesurfaces}
	Let $C$ be the nodal cubic, with $\textup{Pic}^0(C)$ identified with $\bG_m.$ Let $a, q \in \bQ$ be multiplicatively independent. Consider the surface $X_p^q$ given by the blow-up of $\bP^2$ at $z_1 = ... = z_7 = 1$ and $z_8 = a, z_9 = qa^{-1},$ where $z_i$ are points on the nodal cubic. Then at least one of $X^2_p, X^3_p,$ and $X^5_p$ has both a polyhedral effective cone for a set of primes of positive density, and a non-polyhedral effective cone for a set of primes of positive density.
	\end{cor}
	
	\section{Sage Code for Section 2} \label{code}
	\subsection{Code for Section 2, Case 1.} \label{code1}
	The following code finds triangles for which $\frac{a}{m} > \frac{2}{3}.$
	\begin{lstlisting}[language = Python,  basicstyle=\ttfamily\small, showstringspaces=false]
# Check gcd conditions and bounds on a/m
def check(e, x, p, s, k):
    d=(e*s*x+p)/2
    a = e*s^2
    c = e*(s+x)^2
    m = e*s*(s+x)
    return s>0 and gcd(a,d)==1 and (s-x)/s > 2/3 and m-a-d==gcd(2*d+a,c)

def check_integrality(k, y):
    return k.is_integer() and int(k)%2==1 and y.is_integer() and y >= 0

# Loop over parameters (e, x, p)
k, y = var('k,y')
for x in range(1,5):
    print("Solutions for t - s = %s" %x)
    for e in range (1, floor(18/(x^2))):
        for p in range (0, e*x^2):
            rhs = 16*(e*x^2 -p)^2
            divs = divisors(rhs)
            for div in divs:
                eqn1 = (e*x^2*k - 4*(e*x^2-p) - x*y == div)
                eqn2 = (e*x^2*k - 4*(e*x^2-p) + x*y == rhs/div)
                sol = solve([eqn1,eqn2],k,y, solution_dict=True)
                k1 = sol[0][k]
                y1 = sol[0][y]
                if check_integrality(k1, y1):
                    D =  e^2*x^2*(4-k1)^2 -8*e*(2*e*x^2 - p*k1)
                    s1 = 1/(4*e)*(-e*x*(4 -k1) + sqrt(D))
                    s2 = 1/(4*e)*(-e*x*(4 -k1) - sqrt(D))
                    for s in [s1,s2]:
                        if s.is_integer() and check(e,x,p,s,k1): print(sol)
	\end{lstlisting}
	
	\subsection{Code for Section 2, Case 2.}\label{code2}
	The following code finds triangles for which $\half < \frac{a}{m} \leq~\frac{2}{3}.$
	\begin{lstlisting}[language = Python, basicstyle=\ttfamily\small, showstringspaces=false]
# Loop over parameter k
possiblek = [4,5,7,8,9,10, 11]
for k in possiblek:
    divs = divisors(k)
    for x in divs:
        for a in range(1, floor((4*k*x)/(2*k-9))):
            m1 = 1/2*(k*a + sqrt((k*a)^2 - 4*k*a*(a + x)))
            m2 = 1/2*(k*a - sqrt((k*a)^2 - 4*k*a*(a + x)))
            for m in [m1, m2]:
                if m.is_integer() and a/m > 1/2:
                    print("k,a,x,m = %s,%s,%s,%s" %(k,a,x,m))
	\end{lstlisting}
	
	\subsection{Code for Section 2, Case 3.}\label{code3}
	The following code finds triangles for which $ \frac{1}{3}~\leq~\frac{a}{m}~\leq~\half.$
	\begin{lstlisting}[language = Python, basicstyle=\ttfamily\small, showstringspaces=false]
# Check that the lattice triangle is primitive
def primitive(b, c, x, y, m):
    a = (m1-y)/2
    d = m1-a-x
    b = b*d
    c = c*d
    return m>0 and gcd(a, gcd(b, c)) == 1
    
# Loop over parameters (b = b0, c = c0, x, y)
for b in range(3, 12):
    for c in range(b+2, 12):
        divs = divisors(c)
        for x in divs:
            for y in range(0, 2*x):
                D = (2*c*x)^2 + 4*(c-4)*c*y*(y - 2*x)
                m1 = 1/(2*(c-4))*(2*c*x + sqrt(D))
                m2 = 1/(2*(c-4))*(2*c*x - sqrt(D))
                for m in [m1, m2]:
                    if m.is_integer() and gcd(b,c)==1 and primitive(b,c,x,y,m):
                        a = (m-y)/2
                        d = m-a-x
                        if x == gcd(b*d - a, c*d):
                            print("m,a,b,c= %s,%s,%s,%s" %(m,a,b*d,c*d))
	\end{lstlisting}

	
\end{document}

%% file: triangle1
\tikzset{every picture/.style={line width=0.75pt}} 
\begin{tikzpicture}[x=0.75pt,y=0.75pt,yscale=-1,xscale=1]

\draw    (227,48) .. controls (244,60) and (261,61) .. (284,47) ;
\draw    (275,48) .. controls (284,66) and (305,74) .. (332,67) ;
\draw    (318,62) .. controls (316,87) and (323,101) .. (348,106) ;
\draw    (341,100) .. controls (334,113) and (333,131) .. (351,139) ;
\draw [line width=2.25]    (348,132) .. controls (334,146) and (331,161) .. (336,178) ;
\draw    (342,174) .. controls (320,168) and (303,178) .. (299,196) ;
\draw    (264,201) .. controls (276,189) and (292,187) .. (308,192) ;
\draw [line width=2.25]    (191,73) .. controls (208,80) and (234,67) .. (240,46) ;
\draw    (174,106) .. controls (195,101) and (206,80) .. (203,65) ;
\draw    (178,143) .. controls (185,131) and (189,114) .. (179,96) ;
\draw    (193,175) .. controls (198,154) and (192,141) .. (177,131) ;
\draw [line width=2.25]    (186,168) .. controls (208,165) and (224,175) .. (230,196) ;
\draw    (221,191) .. controls (239,185) and (263,186) .. (274,203) ;
\draw [color={rgb, 255:red, 189; green, 16; blue, 224 }  ,draw opacity=1 ]   (179,115) .. controls (217.5,123) and (270.5,159) .. (290,195) ;
\draw [color={rgb, 255:red, 245; green, 166; blue, 35 }  ,draw opacity=1 ]   (260,49) .. controls (271,93) and (302,112) .. (344,122) ;

\draw (213.5,50) node   [align=left] {\begin{minipage}[lt]{21.08pt}\setlength\topsep{0pt}
\mbox{-}1
\end{minipage}};
\draw (357.5,158) node   [align=left] {\begin{minipage}[lt]{21.08pt}\setlength\topsep{0pt}
\mbox{-}1
\end{minipage}};
\draw (202.5,193) node   [align=left] {\begin{minipage}[lt]{21.08pt}\setlength\topsep{0pt}
\mbox{-}1
\end{minipage}};
\draw (181.5,83) node   [align=left] {\begin{minipage}[lt]{21.08pt}\setlength\topsep{0pt}
\mbox{-}2
\end{minipage}};
\draw (251.5,209) node   [align=left] {\begin{minipage}[lt]{21.08pt}\setlength\topsep{0pt}
\mbox{-}2
\end{minipage}};
\draw (173.5,165) node   [align=left] {\begin{minipage}[lt]{21.08pt}\setlength\topsep{0pt}
\mbox{-}2
\end{minipage}};
\draw (307.5,50) node   [align=left] {\begin{minipage}[lt]{21.08pt}\setlength\topsep{0pt}
\mbox{-}2
\end{minipage}};
\draw (356.5,116) node   [align=left] {\begin{minipage}[lt]{21.08pt}\setlength\topsep{0pt}
\mbox{-}3
\end{minipage}};
\draw (340.5,82) node   [align=left] {\begin{minipage}[lt]{21.08pt}\setlength\topsep{0pt}
\mbox{-}2
\end{minipage}};
\draw (333.5,197) node   [align=left] {\begin{minipage}[lt]{21.08pt}\setlength\topsep{0pt}
\mbox{-}2
\end{minipage}};
\draw (261.5,41) node   [align=left] {\begin{minipage}[lt]{21.08pt}\setlength\topsep{0pt}
\mbox{-}3
\end{minipage}};
\draw (169.5,122) node   [align=left] {\begin{minipage}[lt]{21.08pt}\setlength\topsep{0pt}
\mbox{-}3
\end{minipage}};
\draw (297.5,211) node   [align=left] {\begin{minipage}[lt]{21.08pt}\setlength\topsep{0pt}
\mbox{-}3
\end{minipage}};
\draw (305,120) node [anchor=north west][inner sep=0.75pt]   [align=left] {$\tilde{g}$};
\draw (218,101) node [anchor=north west][inner sep=0.75pt]   [align=left] {$\tilde{h}$};

\end{tikzpicture}

%% file: singularfibers
\tikzset{every picture/.style={line width=0.75pt}} 

\begin{tikzpicture}[x=0.75pt,y=0.75pt,yscale=-1,xscale=1]

\draw    (126,150.5) -- (185,150.5) ;
\draw    (167,110.5) -- (226,110.5) ;
\draw    (136,101.5) -- (136,163.5) ;
\draw    (148,101.5) -- (148,163.5) ;
\draw    (175,101.5) -- (175,163.5) ;
\draw    (167,122.5) -- (226,122.5) ;
\draw    (326,110.5) -- (390,110.5) ;
\draw    (336,100.5) -- (336,162.5) ;
\draw    (325,151.5) -- (389,151.5) ;
\draw    (376,99.5) -- (376,161.5) ;
\draw    (545.27,82.5) .. controls (507.12,186.37) and (482.85,131.43) .. (501.05,115.98) ;
\draw    (501.05,115.98) .. controls (515.79,109.11) and (538.33,147.74) .. (547,179.5) ;

\draw (180,195) node [anchor=north west][inner sep=0.75pt]   [align=left] {$I_1^*$};
\draw (347,195) node [anchor=north west][inner sep=0.75pt]   [align=left] {$I_4$};
\draw (511,197) node [anchor=north west][inner sep=0.75pt]   [align=left] {$I_1$};

\end{tikzpicture}

%% file: triangle2
\begin{tikzpicture}[x=0.75pt,y=0.75pt,yscale=-1,xscale=1]

\draw [line width=2.25]    (250,45) .. controls (263,53.5) and (285,53.5) .. (296,43.5) ;
\draw    (360,108.5) .. controls (355,117.5) and (355,137.5) .. (362,150.5) ;
\draw    (184,104.5) .. controls (189,115.5) and (187,132.5) .. (183,143.5) ;
\draw [line width=2.25]    (258,205.86) .. controls (271,198.65) and (291,197.5) .. (304,205.5) ;
\draw    (284,43) .. controls (291,52.5) and (315,62.5) .. (332,60.5) ;
\draw    (322,57.5) .. controls (325,69.5) and (338,85.5) .. (354,88.5) ;
\draw    (344,80.5) .. controls (345,91.5) and (354,109.5) .. (364,114.5) ;
\draw    (219,58.5) .. controls (228,61.5) and (255,55.5) .. (261,43.5) ;
\draw    (195,82.5) .. controls (207,81.5) and (229,67.5) .. (232,54.5) ;
\draw [line width=2.25]    (178,112.5) .. controls (191,111.5) and (204,92.5) .. (205,77.5) ;
\draw    (183,133.5) .. controls (189,137.5) and (200,156.5) .. (199,175.5) ;
\draw    (196,167.5) .. controls (206,168.5) and (226,184.5) .. (225,197.5) ;
\draw    (221,191.5) .. controls (231,192.5) and (256,198.5) .. (264,208.5) ;
\draw    (365,141.5) .. controls (360,146.5) and (347,166.5) .. (351,178.5) ;
\draw    (334,194.5) .. controls (323,189.5) and (303,194.5) .. (296,206.5) ;
\draw    (357,171.5) .. controls (344,174.5) and (331,182.5) .. (327,201.5) ;
\draw [color={rgb, 255:red, 189; green, 16; blue, 224 }  ,draw opacity=1 ]   (210,72) .. controls (230,88.5) and (311,93.5) .. (337,72.5) ;
\draw [color={rgb, 255:red, 245; green, 166; blue, 35 }  ,draw opacity=1 ]   (181,126) .. controls (218,130.5) and (248,179.5) .. (243,205.5) ;

\draw (184,93.25) node   [align=left] {\begin{minipage}[lt]{21.76pt}\setlength\topsep{0pt}
\mbox{-}1
\end{minipage}};
\draw (288,222.25) node   [align=left] {\begin{minipage}[lt]{21.76pt}\setlength\topsep{0pt}
\mbox{-}1
\end{minipage}};
\draw (276,34.25) node   [align=left] {\begin{minipage}[lt]{21.76pt}\setlength\topsep{0pt}
\mbox{-}1
\end{minipage}};
\draw (241,41.75) node   [align=left] {\begin{minipage}[lt]{21.76pt}\setlength\topsep{0pt}
\mbox{-}2
\end{minipage}};
\draw (172,131.25) node   [align=left] {\begin{minipage}[lt]{21.76pt}\setlength\topsep{0pt}
\mbox{-}2
\end{minipage}};
\draw (179,164.75) node   [align=left] {\begin{minipage}[lt]{21.76pt}\setlength\topsep{0pt}
\mbox{-}2
\end{minipage}};
\draw (198,196.75) node   [align=left] {\begin{minipage}[lt]{21.76pt}\setlength\topsep{0pt}
\mbox{-}2
\end{minipage}};
\draw (208,60.75) node   [align=left] {\begin{minipage}[lt]{21.76pt}\setlength\topsep{0pt}
\mbox{-}3
\end{minipage}};
\draw (242,222.61) node   [align=left] {\begin{minipage}[lt]{21.76pt}\setlength\topsep{0pt}
\mbox{-}7
\end{minipage}};
\draw (316,43.75) node   [align=left] {\begin{minipage}[lt]{21.76pt}\setlength\topsep{0pt}
\mbox{-}2
\end{minipage}};
\draw (322,215.75) node   [align=left] {\begin{minipage}[lt]{21.76pt}\setlength\topsep{0pt}
\mbox{-}2
\end{minipage}};
\draw (355,198.75) node   [align=left] {\begin{minipage}[lt]{21.76pt}\setlength\topsep{0pt}
\mbox{-}2
\end{minipage}};
\draw (378,167.25) node   [align=left] {\begin{minipage}[lt]{21.76pt}\setlength\topsep{0pt}
\mbox{-}2
\end{minipage}};
\draw (380,131.25) node   [align=left] {\begin{minipage}[lt]{21.76pt}\setlength\topsep{0pt}
\mbox{-}2
\end{minipage}};
\draw (374,96.75) node   [align=left] {\begin{minipage}[lt]{21.76pt}\setlength\topsep{0pt}
\mbox{-}2
\end{minipage}};
\draw (356,69.75) node   [align=left] {\begin{minipage}[lt]{21.76pt}\setlength\topsep{0pt}
\mbox{-}3
\end{minipage}};
\draw (239,139) node [anchor=north west][inner sep=0.75pt]   [align=left] {$\tilde{g}$};
\draw (272,96) node [anchor=north west][inner sep=0.75pt]   [align=left] {$\tilde{h}$};

\end{tikzpicture}

%% file: triangle3
\begin{tikzpicture}[x=0.75pt,y=0.75pt,yscale=-1,xscale=1]

\draw [line width=2.25]    (250,45) .. controls (263,53.5) and (285,53.5) .. (296,43.5) ;
\draw    (366,112.5) .. controls (361,121.5) and (365,147.5) .. (372,155.5) ;
\draw    (181,99.5) .. controls (186,110.5) and (179,133.5) .. (168,141.5) ;
\draw [line width=2.25]    (228,211.86) .. controls (241,204.65) and (267,211.5) .. (272,220.5) ;
\draw    (284,43) .. controls (291,52.5) and (315,62.5) .. (332,60.5) ;
\draw    (322,57.5) .. controls (325,69.5) and (338,85.5) .. (354,88.5) ;
\draw    (348,83.5) .. controls (349,94.5) and (358,112.5) .. (368,117.5) ;
\draw    (221,56.5) .. controls (230,59.5) and (254,55.5) .. (260,43.5) ;
\draw    (195,77.5) .. controls (207,76.5) and (225,65.5) .. (228,52.5) ;
\draw [line width=2.25]    (175,107.5) .. controls (188,106.5) and (201,87.5) .. (202,72.5) ;
\draw    (169,133.5) .. controls (175,137.5) and (182,156.5) .. (181,175.5) ;
\draw    (178,168.5) .. controls (188,169.5) and (207,189.5) .. (206,201.5) ;
\draw    (201,196.5) .. controls (212,192.5) and (235,205.5) .. (239,215.5) ;
\draw    (372,147.5) .. controls (367,152.5) and (354,172.5) .. (358,184.5) ;
\draw    (344,202.5) .. controls (330,200.5) and (311,211.5) .. (304,220.5) ;
\draw    (364,178.5) .. controls (357,176.5) and (337,195.5) .. (338,206.5) ;
\draw    (310,217.5) .. controls (302,211.5) and (272,213.5) .. (266,221.5) ;
\draw [color={rgb, 255:red, 189; green, 16; blue, 224 }  ,draw opacity=1 ]   (241,49.5) .. controls (258,75.5) and (314,100.5) .. (359,101.5) ;
\draw [color={rgb, 255:red, 245; green, 166; blue, 35 }  ,draw opacity=1 ]   (175,121) .. controls (197,131.5) and (222,175.5) .. (220,205.5) ;

\draw (276,37.25) node   [align=left] {\begin{minipage}[lt]{21.76pt}\setlength\topsep{0pt}
\mbox{-}1
\end{minipage}};
\draw (186,89.25) node   [align=left] {\begin{minipage}[lt]{21.76pt}\setlength\topsep{0pt}
\mbox{-}1
\end{minipage}};
\draw (205,62.25) node   [align=left] {\begin{minipage}[lt]{21.76pt}\setlength\topsep{0pt}
\mbox{-}2
\end{minipage}};
\draw (167,123.75) node   [align=left] {\begin{minipage}[lt]{21.76pt}\setlength\topsep{0pt}
\mbox{-}2
\end{minipage}};
\draw (171,160.75) node   [align=left] {\begin{minipage}[lt]{21.76pt}\setlength\topsep{0pt}
\mbox{-}2
\end{minipage}};
\draw (222,218.25) node   [align=left] {\begin{minipage}[lt]{21.76pt}\setlength\topsep{0pt}
\mbox{-}7
\end{minipage}};
\draw (353,69.75) node   [align=left] {\begin{minipage}[lt]{21.76pt}\setlength\topsep{0pt}
\mbox{-}2
\end{minipage}};
\draw (184,194.75) node   [align=left] {\begin{minipage}[lt]{21.76pt}\setlength\topsep{0pt}
\mbox{-}2
\end{minipage}};
\draw (288,237.25) node   [align=left] {\begin{minipage}[lt]{21.76pt}\setlength\topsep{0pt}
\mbox{-}2
\end{minipage}};
\draw (321,45.25) node   [align=left] {\begin{minipage}[lt]{21.76pt}\setlength\topsep{0pt}
\mbox{-}2
\end{minipage}};
\draw (337,226.75) node   [align=left] {\begin{minipage}[lt]{21.76pt}\setlength\topsep{0pt}
\mbox{-}2
\end{minipage}};
\draw (377,93.75) node   [align=left] {\begin{minipage}[lt]{21.76pt}\setlength\topsep{0pt}
\mbox{-}3
\end{minipage}};
\draw (237,39.75) node   [align=left] {\begin{minipage}[lt]{21.76pt}\setlength\topsep{0pt}
\mbox{-}3
\end{minipage}};
\draw (368,201.75) node   [align=left] {\begin{minipage}[lt]{21.76pt}\setlength\topsep{0pt}
\mbox{-}2
\end{minipage}};
\draw (388,172.25) node   [align=left] {\begin{minipage}[lt]{21.76pt}\setlength\topsep{0pt}
\mbox{-}2
\end{minipage}};
\draw (388,130.75) node   [align=left] {\begin{minipage}[lt]{21.76pt}\setlength\topsep{0pt}
\mbox{-}2
\end{minipage}};
\draw (250,231.25) node   [align=left] {\begin{minipage}[lt]{21.76pt}\setlength\topsep{0pt}
\mbox{-}1
\end{minipage}};
\draw (276,94) node [anchor=north west][inner sep=0.75pt]   [align=left] {$\tilde{h}$};
\draw (224,145) node [anchor=north west][inner sep=0.75pt]   [align=left] {$\tilde{g}$};

\end{tikzpicture}

%% file: surfacexy
\tikzset{every picture/.style={line width=0.75pt}} 

\begin{tikzpicture}[x=0.75pt,y=0.75pt,yscale=-1,xscale=1]

\draw    (123,267.5) .. controls (134,223.5) and (260,196.5) .. (283,213.5) ;
\draw    (167.67,194.75) .. controls (164.8,208.59) and (167.67,227.04) .. (182,230.5) ;
\draw    (157.07,166.37) .. controls (154.2,180.21) and (157.07,198.67) .. (171.4,202.13) ;
\draw    (147.47,137) .. controls (144.6,150.84) and (147.47,169.29) .. (161.8,172.75) ;
\draw    (135.43,110.32) .. controls (132.57,124.16) and (135.43,142.61) .. (149.77,146.07) ;
\draw    (123.4,81.94) .. controls (120.53,95.78) and (123.4,114.24) .. (137.73,117.7) ;
\draw    (112.37,53.87) .. controls (109.5,67.71) and (112.37,86.17) .. (126.7,89.63) ;
\draw    (103.77,28.5) .. controls (100.9,42.34) and (103.77,60.79) .. (118.1,64.25) ;
\draw    (127.7,102.7) .. controls (116.23,98.09) and (91.87,107.32) .. (89,117.7) ;
\draw    (469,185.5) .. controls (480,141.5) and (606,114.5) .. (629,131.5) ;
\draw    (513.67,112.75) .. controls (510.8,126.59) and (513.67,145.04) .. (528,148.5) ;
\draw [shift={(513.67,112.75)}, rotate = 101.7] [color={rgb, 255:red, 0; green, 0; blue, 0 }  ][fill={rgb, 255:red, 0; green, 0; blue, 0 }  ][line width=0.75]      (0, 0) circle [x radius= 3.35, y radius= 3.35]   ;
\draw    (334,150.5) -- (399,150.5) ;
\draw [shift={(401,150.5)}, rotate = 180] [color={rgb, 255:red, 0; green, 0; blue, 0 }  ][line width=0.75]    (10.93,-3.29) .. controls (6.95,-1.4) and (3.31,-0.3) .. (0,0) .. controls (3.31,0.3) and (6.95,1.4) .. (10.93,3.29)   ;
\draw   (49,18) -- (312,18) -- (312,278.5) -- (49,278.5) -- cycle ;
\draw   (435,39.5) -- (676,39.5) -- (676,244.5) -- (435,244.5) -- cycle ;

\draw (105.77,31.5) node [anchor=north west][inner sep=0.75pt]   [align=left] {\mbox{-}2};
\draw (122.77,60.5) node [anchor=north west][inner sep=0.75pt]   [align=left] {\mbox{-}2};
\draw (85.77,84.5) node [anchor=north west][inner sep=0.75pt]   [align=left] {\mbox{-}2};
\draw (153.77,102.5) node [anchor=north west][inner sep=0.75pt]   [align=left] {...};
\draw (167.77,172.5) node [anchor=north west][inner sep=0.75pt]   [align=left] {\mbox{-}2};
\draw (181.77,197.5) node [anchor=north west][inner sep=0.75pt]   [align=left] {\mbox{-}1};
\draw (235.77,220.5) node [anchor=north west][inner sep=0.75pt]   [align=left] {C};
\draw (581.77,138.5) node [anchor=north west][inner sep=0.75pt]   [align=left] {C};
\draw (474.77,100.5) node [anchor=north west][inner sep=0.75pt]   [align=left] {$\mathbb{E}_7$};
\draw (274.77,38.5) node [anchor=north west][inner sep=0.75pt]   [align=left] {X};
\draw (644.77,60.5) node [anchor=north west][inner sep=0.75pt]   [align=left] {Y};

\end{tikzpicture}